%% file: main.tex
\documentclass{siamonline1116}

\input{shared}

\ifpdf
\hypersetup{
  pdftitle={\TheTitle},
  pdfauthor={\TheAuthors}
}
\fi

\begin{document}

\maketitle

\begin{abstract}
We study the iterative solution of linear systems of equations arising from stochastic Galerkin finite element discretizations of saddle point problems.~We focus on the Stokes model with random data parametrized by uniformly distributed random variables. We introduce a Bramble-Pasciak conjugate gradient method as a linear solver.~This method is associated with a block tri\-angular preconditioner which must be scaled using a properly chosen parameter.~We show how the existence requirements of such a conjugate gradient method can be met in our setting.~As a reference solver, we consider a standard MINRES method, which is restricted to symmetric preconditioning.~We analyze the performance of the two different solvers depending on relevant physical and numerical parameters by means of eigenvalue estimates.~For this purpose, we derive bounds for the eigenvalues of the relevant preconditioned sub-matrices.~We illustrate our findings using the flow in a driven cavity as a numerical test case, where the viscosity is given by a truncated Karhunen-Lo\`eve expansion of a random field.~In this example, a Bramble-Pasciak conjugate gradient method with a block triangular preconditioner converges faster than a MINRES method with a comparable block diagonal preconditioner in terms of iteration counts.
\end{abstract}



\input{introduction.tex}

\input{continuous_formulation.tex}

\input{input_modeling_variational_form.tex}
\input{discretization.tex}

\input{preconditioning.tex}

\input{eigenvalue_analysis.tex}

\input{iterative_solvers.tex}

\input{numerical_examples.tex}

\input{conclusion.tex}

\bibliographystyle{siamplain}
\bibliography{references}
\end{document}

%% file: shared.tex
\usepackage{booktabs}
\usepackage{lipsum}
\usepackage{amsfonts}
\usepackage{graphicx}
\usepackage[caption=false]{subfig}
\usepackage{epstopdf}
\usepackage{algorithmic}
\usepackage{enumitem}
\ifpdf
  \DeclareGraphicsExtensions{.eps,.pdf,.png,.jpg}
\else
  \DeclareGraphicsExtensions{.eps}
\fi

\newcommand{\vecg}[1]{\boldsymbol{#1}}    
\newcommand{\dd}{\hspace{0.1cm} \mathrm{d}}

\newcommand{\norm}[1]{\| #1 \|}                 
\newcommand{\abs}[1]{| #1 |}

\newcommand{\RR}{\mathbb{R}}
\newcommand{\NN}{\mathbb N}
\newcommand{\parameterDomain}{\Xi}
\newcommand{\parameter}[1]{y_{#1}}
\newcommand{\parameterVector}{y}

\newcommand{\nPC}{Q}

\newcommand{\nFE}{N}
\newcommand{\iKL}{m}
\newcommand{\nKL}{M}
\newcommand{\spatialDomain}{D}
\newcommand{\spatialDomainBoundary}{\partial D}
\newcommand{\stochasticDomain}{\Omega}
\newcommand{\stochasticDomainValue}{\omega}
\newcommand{\sigmaAlgebra}{\mathcal F}
\newcommand{\probabilityMeasure}{\mathbb P}
\newcommand{\pTuple}{(\stochasticDomain,\sigmaAlgebra,\probabilityMeasure)}
\newcommand{\probabilityDensity}{\rho}
\newcommand{\stochasticDegree}{k}

\numberwithin{theorem}{section}
\numberwithin{equation}{section}

\newcommand{\TheTitle}{A Bramble-Pasciak conjugate gradient method for discrete Stokes equations with random viscosity} 
\newcommand{\ShortTitle}{A BPCG method for discrete Stokes equations with random viscosity}
\newcommand{\TheAuthors}{C. M\"uller, S. Ullmann, and J. Lang}

\headers{\ShortTitle}{\TheAuthors}

\title{{\TheTitle}\thanks{
This work is supported by the Excellence Initiative of the German
federal and state governments and the Graduate School of Computational
Engineering at Technische Universit\"at Darmstadt.}}

\author{
  Christopher M\"uller\thanks{Graduate School Computational Engineering, Technische Universit\"at Darmstadt, 
  Dolivostr. 15, 64293, Darmstadt
    (\email{cmueller@gsc.tu-darmstadt.de}, \email{ullmann@gsc.tu-darmstadt.de}).}
  \and
  Sebastian Ullmann\footnotemark[2]
  \and
  Jens Lang \thanks{Department of Mathematics, Technische Universit\"at Darmstadt, 
  Dolivostr. 15, 64293, Darmstadt (\email{lang@mathematik.tu-darmstadt.de})}
}

\usepackage{amsopn}
\DeclareMathOperator{\diag}{diag}

\DeclareMathOperator*{\essinf}{ess\,inf}
\DeclareMathOperator*{\esssup}{ess\,sup}

%% file: introduction.tex
\section{Introduction}
\label{sec:intro}
We consider a stochastic Galerkin finite element (SGFE) method \cite{BabushkaEtAl2004,GunzburgerEtAl2014} for estimating statistical quantities in the context of Stokes problems with uncertain viscosity. SGFE methods assume a parametrization of the uncertain input in terms of a random vector with a given probability density. This allows a weak formulation of the governing equations in terms of an integral over both the spatial domain and the image domain of the random vector.~An SGFE method is then defined by a Galerkin projection of the weak equations onto a finite dimensional subspace, namely the tensor product of a finite element (FE) space for the spatial dimensions times a multivariate global polynomial space for the stochastic dimensions. This leads to a block structured system of coupled equations, where each block can be interpreted as a finite element discretization of a deterministic problem.

Using global polynomials to model the dependence of the solution on the random data can lead to exponential convergence rates with respect to the number of stochastic degrees of freedom \cite{BabuskaEtAl2007,BespalovEtAl2012}. If the necessary regularity assumptions are met, this can provide an advantage over more robust Monte Carlo-type sampling methods. Among global polynomial-based approaches, stochastic collocation methods \cite{BabuskaEtAl2007,GunzburgerEtAl2014} are currently the most prominent competitors to stochastic Galerkin (SG) methods. Stochastic collocation methods benefit from their ability to estimate statistical quantities by uncoupled solutions of deterministic problems, just like Monte Carlo methods. SG methods, on the other hand, provide the opportunity to derive reliable error estimators and adaptive refinement schemes for the deterministic and stochastic parameter domains by exploiting Galerkin orthogonality \cite{BespalovSilvester2016,EigelEtAl2014}.

The efficient solution of SG systems of equations remains a challenge:~While the block structured system can be decoupled for specific choices of basis functions in some cases \cite{BabushkaEtAl2004}, often a coupled system of equations needs to be solved. Iterative methods are usually employed for this task, because direct methods are not feasible for realistic problem sizes. The inherent ill-conditioning raises the question of efficient preconditioning.

The goal of this study is to derive an efficient preconditioned iterative method that takes advantage of the structural properties of the SGFE system of linear algebraic equations originating from a Stokes problem with random viscosity. Our research builds on results concerning preconditioned iterative solvers for partial differential equations with uncertain data, in particular elliptic problems \cite{PowellElman2009,RosseelVandewalle2010,EUllmann2010} and saddle point problems 
\cite{ErnstEtAl2009,PowellSilvester2012}.

The MINRES method \cite{PaigeSaunders1975} is a standard iterative solver for symmetric indefinite saddle point problems.~We propose a Bramble-Pasciak conjugate gradient (BPCG) method as a promising alternative. Such a method has been introduced in \cite{BramblePasciak1988} for deterministic symmetric saddle point problems.~An attractive feature of BPCG methods is that they allow the use of nonsymmetric preconditioners, which are not applicable within MINRES methods.~The BPCG algorithm, however, relies on a scaling parameter which has to be determined additionally. We derive a BPCG method for Stokes problems with random viscosity by transferring the approach of \cite{PetersEtAl2005} to a stochastic setting. In this context, the dependence of the iteration counts on stochastic discretization and model parameters is of particular importance. We focus on the dimensionality of the random parametrization, the degree of the stochastic polynomial approximation and the magnitude of the random variation. We show that the iteration counts of the proposed preconditioned solvers for the SGFE equations are largely independent on the discretization parameters.~Our numerical experiments confirm these results and reveal that the BPCG solver often converges faster than a standard preconditioned MINRES method in terms of iteration counts.

The outcome of our study is the description and analysis of a BPCG solver for SGFE discretizations of Stokes problems with uniform random viscosity. With some minor adjustments, the solver can be used for Stokes flow with lognormal random viscosity as well \cite{MuellerEtAl2017}. In principle, it is possible to transfer our method to related classes of problems by exchanging individual building blocks. Other potential applications include mixed diffusion problems with random coefficient \cite{ErnstEtAl2009} and constrained optimization problems \cite{AxelssonNeytcheva2003} where the data is uncertain.

The outline of the paper is as follows: We introduce the continuous setting for the Stokes problem with random viscosity in \cref{sec_conti_prob}. We model the viscosity random field in terms of a Karhunen-Lo\`eve expansion in \cref{sec_input_model} and formulate assumptions to ensure that the viscosity is uniformly positive, see \cref{sec_var_form}. In \cref{sec_SGFE_discretization}, we approximate the Stokes problem in finite dimensional spaces with a combination of Taylor-Hood finite elements and Legendre chaos polynomials. We derive a matrix representation of our problem and discuss preconditioning approaches in \cref{sec_preconditioning}. We consider existing strategies from the FE and SG literature as building blocks for our SGFE preconditioners and derive bounds for the eigenvalues of the preconditioned SGFE Laplacian and Schur complements in \cref{sec_ev_analysis}. In \cref{sec_iterative_solvers}, we discuss properties of the MINRES solver and state conditions under which a conjugate gradient (CG) method can be utilized to solve the SGFE problem. Based on the derived eigenvalue bounds, we predict the convergence behavior of the two iterative solvers with respect to different modeling and discretization parameters. The theoretic results and predictions are checked by means of numerical experiments in \cref{sec_numerical_examples}, where we also compare the iterative methods with respect to their performance. The work is eventually summarized and concluded in the final section.

%% file: continuous_formulation.tex
\section{Continuous formulation}
\label{sec_conti_prob}
We establish the framework for the strong form of the Stokes problem in the spatial and stochastic domain.

Let $\spatialDomain \subset \RR^2$ be a bounded domain with sufficiently regular boundary 
$\spatialDomainBoundary$ and spatial coordinates $x = (x_1,x_2)^{\text{T}} \in \spatialDomain$. Further, 
let $\pTuple$ be a probability space, where $\stochasticDomain$ denotes 
the set of elementary events, $\sigmaAlgebra$ is a $\sigma$-algebra on $\stochasticDomain$ and 
$\probabilityMeasure:\sigmaAlgebra \rightarrow [0,1]$ is a probability measure. The Stokes equations describe 
the behavior of a vector-valued velocity field $u = (u_1,u_2)^{\text{T}}$ and a scalar pressure field $p$ 
subject to viscous and external forcing. 

Taking into account uncertainties in the input data, we model the viscosity as a random field
$\nu = \nu(x,\stochasticDomainValue),\, \nu : \spatialDomain \times \stochasticDomain \to \RR$. From a physical point of view, motivating a spatially varying viscosity is rather difficult because the viscosity is a property of the fluid which is subject to transport. Nevertheless, we use this input model to demonstrate the applicability of our methodology in this general setting. Further, note that the physically relevant case of a spatially constant random viscosity is included as a special case, see also \cite{PowellSilvester2012}. Since the uncertainty propagates through the model, the components of the solution are also defined in terms of random fields. Therefore, the strong form of the Stokes equations with uncertain viscosity reads: Find 
$u= u(x,\stochasticDomainValue)$ and $p=p(x,\stochasticDomainValue)$ such that, 
$\probabilityMeasure$-almost surely,
\begin{alignat*}{2}
-\nabla\cdot \big(\nu(x,\stochasticDomainValue) \,  \nabla  u(x,\stochasticDomainValue)\big) + 
\nabla p(x,\stochasticDomainValue)   &= f(x) &\quad& \mathrm{in} \hspace{0.2cm} \spatialDomain 
\times\stochasticDomain, \notag\\
\label{eq_continuous_stoch_problem}
\nabla \cdot u(x,\stochasticDomainValue)   &= 0 &\quad& \mathrm{in} \hspace{0.2cm} \spatialDomain
\times \stochasticDomain, \\
u(x,\stochasticDomainValue)  &= g(x) &\quad& \mathrm{on} \hspace{0.2cm} 
\spatialDomainBoundary \times \stochasticDomain.\notag 
\end{alignat*}
For simplicity, we assume that $f = (f_1,f_2)^{\text{T}}$ is a deterministic volume force and $g = (g_1,g_2)^{\text{T}}$ is 
deterministic Dirichlet boundary data.~Extending the model to stochastic forcing and boundary data would not introduce 
significant difficulties regarding analytical and numerical properties of the problem and is thus omitted.

%% file: input_modeling_variational_form.tex
\section{Input modeling}
\label{sec_input_model}
In the following, we introduce a Karhunen-Lo\`eve expansion (KLE) of the viscosity random field and derive a 
criterion to ensure uniform positivity.

For the viscosity description, we restrict ourselves to second-order random fields, in particular 
$\nu \in L^2(\stochasticDomain,L^2(\spatialDomain))$. Thereby, it is  possible to represent the random 
field $\nu(x,\stochasticDomainValue)$ as a KLE \cite[Theorem 5.28]{LordEtAl2014} of the form
\begin{equation}
\label{eq_inf_KL}
\nu(x,\stochasticDomainValue) = \nu_0(x) + \sigma_\nu \sum_{\iKL=1}^\infty \sqrt{\lambda_\iKL}
\nu_\iKL(x) \parameter{\iKL}(\stochasticDomainValue),
\end{equation}
where $\nu_0(x)$ is the mean field of the viscosity, 
i.e.~$\nu_0(x)=\int_{\stochasticDomain}\nu(x,\stochasticDomainValue) \dd 
\probabilityMeasure(\stochasticDomainValue)$ is 
the expected value of $\nu(x,\stochasticDomainValue)$. Further, the pairs 
$(\lambda_\iKL,\nu_\iKL)_{\iKL=1}^\infty$ are eigenvalues and eigenvectors of the integral operator associated with 
the covariance of the correlated random field. Finally, the stochastic parameters $\parameter{\iKL},\iKL=1,\dots, \infty$, 
are uncorrelated random variables with zero mean and unit variance. We assume that these random variables are 
stochastically independent. For further properties of the KLE, see 
\cite[sections 5.4 and~7.4]{LordEtAl2014}.

To show well-posedness of saddle point problems, it is required that the elliptic coefficient is positive and 
bounded. For simplicity, we state uniform positivity and boundedness using constants 
$\underline{\nu},\overline{\nu} >0$, independent of $x$ and $\stochasticDomainValue$, such that
\begin{equation}
\label{eq_assumption_viscosity}
0 < \underline{\nu} \leq \nu(x,\stochasticDomainValue) \leq \overline{\nu} < \infty \quad 
\text{a.e. in } \spatialDomain \times
\stochasticDomain.
\end{equation}  
In the following, we make assumptions on the components of the expansion \cref{eq_inf_KL} to guarantee 
that \cref{eq_assumption_viscosity} holds. First of all, 
we assume that the random variables are bounded and uniformly distributed in the infinite dimensional cube 
$\parameterDomain:=[-\sqrt{3},\sqrt{3}]^{\NN}$. We collect the random variables in a vector 
$\parameterVector:=(\parameter{\iKL}(\stochasticDomainValue))_{\iKL\in \NN},\,\parameterVector:\,\stochasticDomain \to \parameterDomain$, 
such that $\norm{\parameter{\iKL}}_{L^\infty(\stochasticDomain)} = \sqrt{3}$ for $\iKL \in \NN$.
Furthermore, we assume that the mean field $\nu_0(x)$ is a bounded positive function, i.e.~there exists a constant $\underline{\nu}_0:=\essinf_{x \in \spatialDomain} \nu_0(x)>0$ and a constant
$\overline{\nu}_0:=\esssup_{x \in \spatialDomain} \nu_0(x)<\infty$. Moreover, let 
$\sqrt{\lambda_\iKL}\nu_\iKL(x) \in L^\infty(\spatialDomain)$, $\chi_\iKL:=\norm{\sqrt{\lambda_\iKL}\nu_\iKL}_{L^\infty(\spatialDomain)}$ for all $m=1,\dots,\infty$, and assume that the sequence $(\chi_\iKL)_{\iKL} \in \ell^1(\NN)$, $\chi:=\norm{(\chi_\iKL)_{\iKL}}_{\ell^1(\NN)}$. From these assumptions and \cref{eq_inf_KL}, we derive
\begin{equation}
\label{eq_KL_estiamte}
\underline{\nu}_0 - \sqrt{3} \, \sigma_\nu \, \chi \leq \nu(x,\stochasticDomainValue) \leq 
\overline{\nu}_0 + \sqrt{3} \, \sigma_\nu \, \chi,
\end{equation}
for all $\stochasticDomainValue \in \stochasticDomain$. Uniform positivity of $\nu(x,\stochasticDomainValue)$ and consequently \cref{eq_assumption_viscosity} follow under the assumption
\begin{equation}
\label{eq_assumption_KL}
\underline{\nu}_0 > \sqrt{3} \, \sigma_\nu \, \chi
\end{equation}
with $\underline{\nu} = \underline{\nu}_0 - \sqrt{3} \, \sigma_\nu \, \chi >0$ and $\overline{\nu} = \overline{\nu}_0 + \sqrt{3} \, \sigma_\nu \, \chi$. Referring to \cref{eq_assumption_KL}, it is thus not possible to model an input viscosity with an arbitrarily large variance for any given mean field.

When the series representation is used in a discrete setting in \cref{sec_SGFE_discretization}, the infinite sum in 
\cref{eq_inf_KL} is truncated after $\nKL$ terms, i.e.
\begin{equation}
\label{eq_truncated_KL}
\nu(x,\stochasticDomainValue) \approx \nu_\nKL(x,\stochasticDomainValue) = \nu_0(x) + 
\sigma_\nu \sum_{\iKL=1}^\nKL \sqrt{\lambda_\iKL}
\nu_\iKL(x) \parameter{\iKL}(\stochasticDomainValue).
\end{equation}
For a given $\nKL$, the truncation error depends on the correlation length and the smoothness of the covariance operator \cite[sections 5.4]{LordEtAl2014}. Assumption \cref{eq_assumption_KL}, stated for the infinite case, naturally also holds for the finite one.

Although we use the KLE \cref{eq_inf_KL} to describe the uncertain input, the methods presented in 
this work are not limited to this representation. They can be applied to any parametrized random field that 
fulfills the positivity and boundedness constraint \cref{eq_assumption_viscosity}.
We will formulate our problem with respect to the deterministic range space 
$\parameterDomain$ instead of $\stochasticDomain$. In the truncated case, the range space is the $\nKL$-dimensional cube 
$\parameterDomain^\nKL:=\parameterDomain_1 \times \ldots \times \parameterDomain_\nKL$, $\parameterDomain_\iKL:=[-\sqrt{3},\sqrt{3}]$.

Due to the assumed stochastic independence of the random variables, the uniform measure on the infinite dimensional cube can 
be expressed by a tensor product of univariate uniform measures, i.e.\ $\probabilityDensity:=\bigotimes_{\iKL=1}^{\infty} \probabilityDensity_\iKL$ with $\probabilityDensity_\iKL(\mathrm{d} \parameter{\iKL}) = \mathrm{d} \parameter{\iKL}/(2\sqrt{3})$.

As a consequence of the Doob-Dynkin lemma \cite[Lemma 2.1.2]{Oksendal1998}, the velocity and pressure random fields can be 
pa\-ra\-me\-tri\-zed with the vector $\parameterVector$ as well. 
\section{Variational formulation}
\label{sec_var_form}
We establish a variational formulation of the Stokes problem with random data and discuss existence and uniqueness 
of weak solutions.

In order to formulate the weak equations, we use Bochner spaces 
$L^2_{\probabilityDensity}(\parameterDomain;X)$, 
where $X$ is a given Banach space.~We consider spaces consisting of all equivalence classes of strongly measurable functions $v \, : \, \parameterDomain \to X$ with norm 
\begin{equation}
\label{eq_Bochner_norm}
\norm{v}_{L^2_{\probabilityDensity}(\parameterDomain;X)} = 
\left(\int_{\parameterDomain} \norm{v(\cdot,\parameterVector)}_X^2 \, \probabilityDensity(\mathrm{d}\parameterVector) \right)^{1/2} 
< \infty.
\end{equation}
As $X$ will always be a separable Hilbert space in our work, the Bochner spaces are isomorphic to 
the corresponding tensor product spaces $L^2_{\probabilityDensity}(\parameterDomain) \otimes X$ with 
norms $\norm{\cdot}_{L^2_{\probabilityDensity}(\parameterDomain) \otimes X} := 
\norm{\cdot}_{L^2_{\probabilityDensity}(\parameterDomain;X)}$.

We mark function spaces containing vector-valued elements with bold symbols. Recalling the function spaces used for enclosed Stokes flow with deterministic data \cite[section 3.2]{John2016}, we utilize $\vecg{V}_0:=\{v\in \vecg{H}^1(\spatialDomain) \, | \, v_|{_{\spatialDomainBoundary}} = 0 \}$ for the vector-valued velocity field. As a corresponding norm, we use $\norm{v}_{\vecg{V}_0}:= \norm{\nabla v}_{\vecg{L}^2(\spatialDomain)}$ 
for $v \in \vecg{V}_0$. For the scalar pressure field, we use the Sobolev space 
$W_0:=\{q \in L^2(\spatialDomain) \, | \, \smallint_\spatialDomain q(x) \dd x = 0 \}$ with norm 
$\norm{q}_{W_0}:= \norm{q}_{L^2(\spatialDomain)}$ for $q \in W_0$. The inhomogeneous Dirichlet conditions are homogenized using the function $u_0$ which is a lifting of the boundary data $g$ in the sense of the trace theorem.

Now, we introduce the tensor product spaces $\vecg{\mathcal{V}}_0 := L_{\probabilityDensity}^2(\parameterDomain) \otimes \vecg{V}_0$ and 
$\mathcal{W}_0 := L_{\probabilityDensity}^2(\parameterDomain) \otimes W_0$, 
which are denoted by calligraphic letters and which are Hilbert spaces as well. Further, we define a variational 
formulation on these product spaces: Find $(u,p) \in \vecg{\mathcal{V}}_0 \times \mathcal{W}_0$ such that 
\begin{equation}
\label{eq_weak_formulation_ps}
\begin{aligned}
\langle a(u,v) \rangle + \langle b(v,p) \rangle &= 
\langle l(v) \rangle, &\quad& 
\forall v \in \vecg{\mathcal{V}}_0, \\
\langle b(u,q) \rangle &= \langle t(q) \rangle, &\quad& \forall q \in \mathcal{W}_0,
\end{aligned}
\end{equation}
where the bilinear forms and linear functionals are given by
\begin{alignat*}{3}
\langle a(u,v) \rangle &:= \int_\parameterDomain \int_{\spatialDomain}\nu(x,\parameterVector) \, \nabla u(x,\parameterVector) \cdot \nabla v(x,\parameterVector)  \dd x \, \probabilityDensity(\mathrm{d}\parameterVector), &\quad& u,v \in \vecg{\mathcal{V}}_0, \\
\langle b(v,q) \rangle &:= - \int_\parameterDomain \int_{\spatialDomain} q(x,\parameterVector) \, \nabla \cdot v(x,\parameterVector) \dd x\, \probabilityDensity(\mathrm{d}\parameterVector), &\quad& q\in \mathcal{W}_0, \, v \in  \vecg{\mathcal{V}}_0, \\
\langle l(v) \rangle &= \int_\parameterDomain\int_{\spatialDomain} f(x) \cdot v(x,\parameterVector) \dd x \, \probabilityDensity(\mathrm{d}\parameterVector) - \langle a(u_0,v)\rangle, &\quad& v \in \vecg{\mathcal{V}}_0, \\
\langle t(q) \rangle &= - \langle b(u_0,q) \rangle, &\quad& q \in \mathcal{W}_0.
\end{alignat*}
Well-posedness of the weak formulation \cref{eq_weak_formulation_ps} can be established using the theory of mixed variational problems in \cite{BrezziFortin1991}, see also \cite[Theorems 2.10 and 2.11]{Mueller2018}. The line of argument is similar when the mixed diffusion problem with uncertain input data is considered, see \cite{BespalovEtAl2012,ErnstEtAl2009}.

%% file: discretization.tex
\section{Stochastic Galerkin finite element discretization}
\label{sec_SGFE_discretization}
We introduce the discretizations of the tensor product spaces and establish the structure of the emerging system of equations.

\subsection{Discrete subspaces}
%
We use FE spaces 
$\vecg{V}_0^h \subset \vecg{V}_0$ and $W_0^h \subset W_0$ to discretize the spatial domain, where $h$ 
denotes the spatial mesh size. In particular, we choose stable Taylor-Hood $P_2/P_1$ elements on a regular 
triangulation \cite[section 3.6.2]{John2016}. They consist of $\nFE_u$ continuous piecewise quadratic basis functions for the velocity space and $\nFE_p$ continuous piecewise linear basis functions for the pressure space. The total number of FE degrees of freedom is denoted by $\nFE:=(\nFE_u + \nFE_p)$.

We truncate the KLE \cref{eq_inf_KL} after $\nKL$ terms so that the random parameter domain reduces to the finite dimensional cube $\parameterDomain^\nKL$, see \cref{eq_truncated_KL}. We define a stochastic Galerkin space $S^{\stochasticDegree} \subset L_{\probabilityDensity}^2(\parameterDomain^\nKL)$ by a complete $\nKL$-variate Legendre chaos of total degree $k$. This is the 
appropriate basis for the chosen uniform input distribution according to the Wiener-Askey scheme 
\cite{XiuKarniadakis2002b}, because the basis functions are orthonormal with respect to the measure $ \probabilityDensity$.~The associated number of degrees of freedom is $\nPC:=\binom{\nKL + \stochasticDegree}{\stochasticDegree}$.

Finally, we define the SGFE spaces $\vecg{\mathcal{V}}_0^{\stochasticDegree h}:= S^{\stochasticDegree} \otimes \vecg{V}_0^h \subset \vecg{\mathcal{V}}_0$ and $\mathcal{W}_0^{\stochasticDegree h}:= S^{\stochasticDegree} \otimes W_0^h \subset \mathcal{W}_0$. The emerging discrete system of equations is of size $\text{dim}(\vecg{\mathcal{V}}_0^{\stochasticDegree h} \times \mathcal{W}_0^{\stochasticDegree h}) =  \nPC\nFE$. Existence and uniqueness of a solution to the fully-discrete variational problem on the SGFE spaces again follows from the theory of mixed finite element problems \cite[Theorem II.1.1]{BrezziFortin1991}. For the mixed  diffusion equations with uncertain input data, the analysis can be found in \cite[Lemma 3.1]{BespalovEtAl2012}.

\subsection{Matrix formulation}
We represent the solution fields and test functions in the finite element and stochastic Galerkin bases and insert them into the weak formulation \cref{eq_weak_formulation_ps}. Together with the truncated KLE \cref{eq_truncated_KL}, this results in a matrix equation of the form
\begin{equation}
\label{eq_block_matrix_form}
\mathcal{C} \, \vecg{z} = \vecg{b}, \quad \mathcal{C} \in \RR^{\nPC \nFE \times \nPC \nFE},\quad\text{where}\quad
\mathcal{C}:= \begin{bmatrix} \mathcal{A} & \mathcal{B}^{\text{T}} \\ \mathcal{B} & 0 \end{bmatrix}, \quad 
\vecg{z}:= \begin{bmatrix} \boldsymbol{u} \\ \boldsymbol{p} \end{bmatrix}, \quad 
\vecg{b}:= \begin{bmatrix} \vecg{f} \\ \vecg{t} \end{bmatrix}.
\end{equation}
Here, $\boldsymbol{u} \in \RR^{\nPC \nFE_u}$ and $\boldsymbol{p} \in \RR^{\nPC \nFE_p}$ are the
discrete velocity and pressure coefficient vectors, respectively. In order to distinguish objects on different spaces, we denote the matrices on the product spaces with calligraphic 
capital letters and the matrices on the FE and SG spaces with
standard capital letters. The sub-matrices are given by 
\begin{alignat}{5}
\label{eq_SGFEM_diffusion_matrix}
\mathcal{A} &= I \otimes A_0 + \sum_{\iKL=1}^{\nKL} G_{\iKL} \otimes A_{\iKL} \, &\in \RR^{\nPC  \nFE_u \times \nPC \nFE_u},&\qquad& \boldsymbol{f} &= \boldsymbol{g}_0 \otimes \boldsymbol{w} \, &\in \RR^{\nPC \nFE_u},\\
\label{eq_SGFEM_gradient_matrix}
\mathcal{B} &= I \otimes B &\in \RR^{\nPC \nFE_p \times \nPC \nFE_u},&\qquad& \boldsymbol{t} &= \boldsymbol{g}_0 \otimes \boldsymbol{d} &\in \RR^{\nPC \nFE_p},
\end{alignat}
where the identity matrix $I \in \RR^{\nPC \times \nPC}$ is a consequence of the orthonormality of the SG basis. 
The matrix $A_0 \in \RR^{\nFE_u \times \nFE_u}$ results from weighting the FE velocity Laplacian with the mean viscosity field 
$\nu_0(x)$. In the same manner, the matrices $A_\iKL \in \RR^{\nFE_u \times \nFE_u}$, 
$\iKL=1,\dots,\nKL$ are FE velocity Laplacians weighted with the fluctuation parts $\sigma_\nu \sqrt{\lambda_\iKL}
\nu_\iKL(x)$, $\iKL=1,\dots,\nKL$, of \cref{eq_truncated_KL}. Further, the entries of the SG matrices $G_\iKL \in \RR^{\nPC \times \nPC}$ are expectations of products of two SG basis functions and one random variable $y_\iKL$, $\iKL=1,\dots,\nKL$. 
Recalling the structure of the KLE in \cref{eq_truncated_KL}, we notice that the 
SGFE Laplacian $\mathcal{A}$ in \cref{eq_SGFEM_diffusion_matrix} is similarly composed of a mean part -- 
namely $I \otimes A_0$ -- and a sum representing fluctuations around it.

%% file: preconditioning.tex
\section{Preconditioning}
\label{sec_preconditioning}
In order to improve the condition of the involved operators and accelerate the convergence of iterative solvers, 
we discuss different preconditioning strategies.~We use existing approaches from the FE and SG literature as building blocks for our preconditioners and restate corresponding properties and results.

The SGFE matrix  $\mathcal{C} \in \RR^{\nPC\nFE\times\nPC\nFE}$ in 
\cref{eq_block_matrix_form} has a symmetric saddle point structure, like its deterministic FE 
counterpart of size $\nFE\times\nFE$, see \cite[section 3.5]{ElmanEtAl2014}. Due to this structural analogy, it seems 
natural to transfer existing methods for deterministic saddle point problems to a stochastic context. This approach is also followed for other SGFE problems such as mixed diffusion \cite{ErnstEtAl2009} 
and the Navier-Stokes equations \cite{PowellSilvester2012}.

Introducing the Schur complement $\mathcal{S}:= -\mathcal{B}  \mathcal{A}^{-1} \mathcal{B}^{\text{T}}$, we can 
formulate the following factorizations of the SGFE system matrix $\mathcal{C}$:
\begin{equation}
\label{eq_block_factorizations}
\begin{bmatrix} \mathcal{A} & \mathcal{B}^{\text{T}} \\ \mathcal{B} & 0 \end{bmatrix} =
\begin{bmatrix} I & 0 \\ \mathcal{B A}^{-1} & I \end{bmatrix}
\begin{bmatrix} \mathcal{A} & 0 \\ 0 & \mathcal{S} \end{bmatrix}
\begin{bmatrix} I & \mathcal{A}^{-1} \mathcal{B}^{\text{T}} \\ 0 & I \end{bmatrix}
= \begin{bmatrix} \mathcal{A} & 0 \\ \mathcal{B}  & \mathcal{S} \end{bmatrix}
\begin{bmatrix} I & \mathcal{A}^{-1} \mathcal{B}^{\text{T}} \\ 0 & I \end{bmatrix}.
\end{equation}
The SGFE Laplacian $\mathcal{A}$ is positive definite due to the uniformly positive bounds on the viscosity field in \cref{eq_assumption_viscosity}, see also \cite[Theorem 3.3]{Mueller2018}. The congruence 
transform in \cref{eq_block_factorizations} then implies that the discrete SGFE problem is highly indefinite.

As a starting point for constructing preconditioners, we consider block diagonal and block triangular 
matrix structures motivated by the factorizations \cref{eq_block_factorizations}.~They are established 
approaches in the context of saddle point problems, see \cite{BenziEtAl2005}. In the setting of this work, such preconditioners rely on appropriate approximations of the SGFE matrices $\mathcal{A}$ and $\mathcal{S}$. These approximations are denoted by $\widetilde{\mathcal{A}}$ and $\widetilde{\mathcal{S}}$ in the following, respectively. The generic preconditioners are of the form
\begin{equation}
\label{eq_precon_structure}
\mathcal{P}_{\text{diag}} = \begin{bmatrix} \widetilde{\mathcal{A}} & 0 \\ 0 & \widetilde{\mathcal{S}}\end{bmatrix}, 
\quad
\mathcal{P}_{\text{tri}} = \begin{bmatrix} \widetilde{\mathcal{A}} & 0 \\ 
\mathcal{B} & \widetilde{\mathcal{S}}\end{bmatrix}.
\end{equation}
In addition to \cref{eq_precon_structure}, we also make a structural simplification: 
We choose the SGFE preconditioners to be Kronecker products of one SG and one FE 
matrix (see also \cite{ErnstEtAl2009,PowellElman2009,PowellSilvester2012}):
\begin{equation}
\label{eq_precon_ansatz_diff}
\widetilde{\mathcal{A}} := \widetilde{G}_A \otimes \widetilde{A}, \quad 
\widetilde{\mathcal{S}} := \widetilde{G}_S \otimes \widetilde{S},
\end{equation}
where $\widetilde{A} \in \RR^{\nFE_u \times \nFE_u}$ and 
$\widetilde{S}\in \RR^{\nFE_p \times \nFE_p}$ shall be approximations of the FE Laplacian and Schur 
complement, respectively. The matrices $\widetilde{G}_A$ and $\widetilde{G}_S$ are approximations of the SG 
matrices. 
When evaluating $\widetilde{\mathcal{A}}^{-1} \mathcal{A}$, one sees that $\widetilde{G}_A$ only acts on the SG matrices $I$ and $G_\iKL$, and $\widetilde{A}$ only acts on the FE matrices $A_0$ and $A_\iKL$, $\iKL=1,\dots,\nKL$. The eigenvalue 
analysis in \cref{subsec_precon_Schur} reveals that this works similarly for the preconditioned SGFE 
Schur complement.

Having fixed the structures in \cref{eq_precon_structure,eq_precon_ansatz_diff}, the preconditioning task is reduced 
to choosing appropriate $\widetilde{A},\widetilde{S},\widetilde{G}_A$ and $\widetilde{G}_S$. 
We do this step by step using established preconditioners from the FE and SG literature.

\subsection{Finite element matrices}
\label{subsec_precon_fe_matrices}
For the discrete Stokes equations with deterministic data, it is known that one multigrid V-cycle, denoted by 
$\widetilde{A}_{\text{mg}}$ in the following, is spectrally equivalent to the FE Laplacian 
$A$, see e.g.~\cite[section 2.5]{ElmanEtAl2014}. This means, there exist positive constants 
$\delta$ and $\Delta$, independent of $h$, such that
\begin{equation}
\label{eq_upper_bound_Diff_approxDiff}
\delta \leq \frac{\vecg{v}^{\text{T}}A\vecg{v}}{\vecg{v}^{\text{T}}\widetilde{A}_{\text{mg}}\vecg{v}} \leq \Delta, \quad
\forall \vecg{v} \in \RR^{\nFE_u}\backslash \{\vecg{0}\}.
\end{equation}
As $A_0$ and $A_\iKL$, $\iKL=1,\dots,\nKL$ are FE Laplacians weighted with $\nu_0(x)$ and $\sigma_\nu \sqrt{\lambda_\iKL}\nu_\iKL(x)$, we can deduce 
\begin{equation}
\label{eq_mean_Diff_Diff}
0 < \underline{\nu}_0 \leq \frac{\vecg{v}^{\text{T}}A_0\vecg{v}}
{\vecg{v}^{\text{T}}A\vecg{v}} \leq \overline{\nu}_0, \quad 
- \sigma_\nu \chi_\iKL  \leq 
\frac{\vecg{v}^{\text{T}}A_\iKL \vecg{v}}{\vecg{v}^{\text{T}}A \vecg{v}} 
\leq  \sigma_\nu \chi_\iKL, 
\quad
\forall \vecg{v} \in \RR^{\nFE_u}\backslash \{\vecg{0}\},
\end{equation}
from the assumptions made in \cref{sec_input_model}. A combination of 
\cref{eq_upper_bound_Diff_approxDiff,eq_mean_Diff_Diff} then yields
\begin{equation}
\label{eq_upper_bound_mean_Diff_approxDiff}
\underline{\nu}_0 \, \delta \leq \frac{\vecg{v}^{\text{T}}A_0\vecg{v}}
{\vecg{v}^{\text{T}} \widetilde{A}_{\text{mg}} \vecg{v}} \leq \overline{\nu}_0 \, \Delta, \quad
- \sigma_\nu \, \chi_\iKL \, \Delta \leq \frac{\vecg{v}^\text{T} A_m \vecg{v}}{\vecg{v}^\text{T} 
\widetilde{A}_{\text{mg}} \vecg{v}} \leq  \sigma_\nu \, \chi_\iKL \, \Delta, \quad
\forall \vecg{v} \in \RR^{\nFE_u}\backslash \{\vecg{0}\}.
\end{equation}
By the help of the bounds \cref{eq_upper_bound_mean_Diff_approxDiff} we now know 
how the multigrid preconditioner effects the spectrum of the FE matrices contained in the SGFE Laplacian 
$\mathcal{A}$. In particular, we note that the mesh size $h$ does not appear in the bounds.

For the Schur complement, we again start from an established result for enclosed Stokes flow: The 
pressure mass matrix $M_p$ is spectrally equivalent to the negative FE Schur complement 
\cite[Theorem 3.22]{ElmanEtAl2014}:
\begin{equation}
\label{eq_lower_bound_Schur_Mass}
\gamma^2 \leq \frac{\vecg{q}^{\text{T}}BA^{-1}B^{\text{T}}\vecg{q}}{\vecg{q}^{\text{T}}M_p \, \vecg{q}}
\leq 1, \quad
\forall \vecg{q} \in \RR^{\nFE_p}\backslash \{\vecg{0},\vecg{1}\}.
\end{equation}
By writing $\backslash\{\vecg{0},\vecg{1}\}$ we mean that all constant vectors $\vecg{q}$ are excluded. This is necessary because they 
are in the nullspace of $B^{\text{T}}$ for enclosed flow \cite[Section 3.3]{ElmanEtAl2014}. Further, $\gamma$ is the inf-sup constant of the chosen 
mixed FE approximation.

The pressure mass matrix is spectrally equivalent to its diagonal $D_p := \diag(M_p)$,
\begin{equation}
\label{eq_bound_Mass_diagMass}
\theta^2 \leq \frac{\vecg{q}^{\text{T}}M_p\vecg{q}}{\vecg{q}^{\text{T}}D_p \, \vecg{q}} \leq \Theta^2, \quad
\forall \vecg{q} \in \RR^{\nFE_p}\backslash \{\vecg{0}\},
\end{equation}
with positive constants $\theta^2$ and $\Theta^2$ depending only on the degree and type of finite elements 
applied \cite{Wathen1991}. For linear basis functions on triangles, as used for the pressure 
approximation in this work, $\theta^2 = \frac12$ and $\Theta^2=2$. Combining 
\cref{eq_lower_bound_Schur_Mass} and \cref{eq_bound_Mass_diagMass} results in the relation
\begin{equation}
\label{eq_bound_Schur_diagMass}
\theta^2\, \gamma^2 \leq \frac{\vecg{q}^{\text{T}}BA^{-1}B^{\text{T}}\vecg{q}}{\vecg{q}^{\text{T}}
D_p \, \vecg{q}} 
\leq \Theta^2, \quad\forall \vecg{q} \in \RR^{\nFE_p}\backslash \{\vecg{0},\vecg{1}\}.
\end{equation}
The bound \cref{eq_bound_Schur_diagMass} reveals that preconditioning the negative FE Schur complement with 
the diagonal of the pressure mass matrix results in eigenvalue bounds independent of $h$. This result 
together with \cref{eq_upper_bound_mean_Diff_approxDiff} is the 
starting point for analyzing the parameter dependence of the eigenvalues of the preconditioned SGFE matrices in \cref{sec_ev_analysis}.

Besides their spectral properties, the considered FE preconditioners are also attractive from a 
computational point of view as they can be applied with linear complexity. For the mentioned reasons, we use 
the multigrid V-cycle and the diagonal of the pressure mass matrix as FE building blocks for our SGFE 
preconditioners \cref{eq_precon_ansatz_diff} such that $\widetilde{A}: = \widetilde{A}_{\text{mg}}$ and $\widetilde{S}: = D_p$.

\subsection{Stochastic Galerkin matrices}
\label{subsec_precon_sg_matrices}
Spectral properties and complexity considerations are also the basis for choosing SG preconditioners. Independence 
of the spectral bounds from the discretization parameters $\stochasticDegree$ and 
$\nKL$ is particularly desired. Because uniform random variables are used to model the uncertain viscosity, this independence comes 
for free: The eigenvalues of $G_\iKL$, $\iKL=1,\dots,\nKL$, are bounded by the support of the 
random variables independently of the chaos degree $k$ \cite[Theorem 9.62 and Corollary 
9.65]{LordEtAl2014}. Furthermore, the number of terms kept in the truncated KLE does not influence the 
location of these eigenvalues but only their multiplicities \cite[Corollary 9.64]{LordEtAl2014}. Using these 
results, we can state the bounds 
\begin{equation}
\label{ev_bounds_G_m}
-\sqrt{3} \leq \frac{\vecg{a}^\text{T} G_\iKL \vecg{a}}{\vecg{a}^\text{T}\vecg{a}}
\leq \sqrt{3}, \quad \forall \vecg{a} \in \RR^{Q} \backslash \{\vecg{0}\},\quad \iKL \in \NN.
\end{equation}
Preconditioners for SGFE matrices often rely on the so-called mean-based approximation~
\cite{PowellElman2009}. As the name suggests, it approximates the mean part of the SGFE problem. In the setting of this work, the associated mean SG matrix is the identity. Using the identity as an SG preconditioner works well for SG matrices originating from uniform random variables \cite{PowellElman2009,PowellSilvester2012} and is often the most efficient 
choice.~This is a conclusion drawn in \cite{RosseelVandewalle2010}, where a more elaborate comparison of 
different SGFE preconditioners is conducted. Thus, we also rely on the mean-based 
approximation in our work, leading to the definition $\widetilde{G}_A = \widetilde{G}_S = I$.

When the fluctuation parts of the SGFE problem become more important, e.g.~when the ratio 
$\sigma_\nu/\nu_0(x)$ increases, using only the mean information for the preconditioner might be 
insufficient. Then, one can resort to approaches which rely on higher moment information, see 
e.g.~\cite{EUllmann2010}.~Although such methods could be implemented in our framework, 
they are computationally more expensive. As the size of the fluctuations 
is limited by \eqref{eq_assumption_KL} and the mean-based preconditioner is very cheap to apply, we only use this 
approach in our work.

%% file: eigenvalue_analysis.tex
\section{Analysis of the preconditioned SGFE matrices}
\label{sec_ev_analysis}
We derive bounds on the eigenvalues of three different matrices: the preconditioned SGFE Laplacian, the preconditioned SGFE Schur complement and the preconditioned approximate SGFE Schur complement.~These bounds can be used as ingredients to construct eigenvalue bounds for the block diagonal and block triangular preconditioned system matrix utilizing  results from the theory of saddle point problems.

As FE preconditioners, we decided to use one multigrid V-cycle $\widetilde{A}=\widetilde{A}_{\text{mg}}$ and 
the diagonal of the pressure mass matrix $\widetilde{S}=D_p$. To precondition the SG matrices, we use 
the approximation $\widetilde{G}_A=\widetilde{G}_S=I$.
In the following, we consider two preconditioners with block structures according to \cref{eq_precon_structure} and 
identical building blocks:
\begin{equation}
\label{eq_block_precon}
\mathcal{P}_1 = \begin{bmatrix} \widetilde{\mathcal{A}}_{\text{mg}} & 0 \\
0 & \widetilde{\mathcal{S}}_p \end{bmatrix}, 
\quad
\mathcal{P}_2 = \begin{bmatrix} a\,\widetilde{\mathcal{A}}_{\text{mg}} & 0 \\ 
\mathcal{B} & -\widetilde{\mathcal{S}}_p \end{bmatrix},
\quad\text{with}\quad
\widetilde{\mathcal{A}}_{\text{mg}}:= I \otimes \widetilde{A}_{\text{mg}},\quad
\widetilde{\mathcal{S}}_p:= I \otimes D_p.
\end{equation}
The scaling parameter $a \in \RR$ and the negative sign of the Schur complement approximation are additional manipulations to the preconditioner $\mathcal{P}_2$ in \cref{eq_block_precon}. Although it is not immediately obvious, they have beneficial effects on the spectrum of the respective preconditioned system matrix. More details are given in 
\cref{subsec_is_bpcg}.

\input{ea_Laplacian.tex}
  
\input{ea_Schur_complement.tex}

\input{ea_pseudo_Schur_complement.tex}

%% file: ea_Laplacian.tex
\subsection{The preconditioned SGFE Laplacian}
The eigenvalue estimates for the preconditioned SGFE Laplacian $\widetilde{\mathcal{A}}_{\text{mg}}^{-1} \, \mathcal{A}$ 
can be derived as in \cite{ErnstEtAl2009}, where the procedure is carried 
out for a preconditioned SGFE matrix in the context of a mixed diffusion problem with 
uncertain coefficient. The main 
difference between our approach and the one in \cite[Lemma 4.6]{ErnstEtAl2009} is that we use a 
product estimate of the Rayleigh quotient similar to \cite{Pultarova2015} to 
derive a lower bound for the eigenvalues of the preconditioned SGFE Laplacian. This bound is tighter than the one 
we would get by following the steps in \cite{ErnstEtAl2009}.
\begin{lemma}
\label{lemma_eigenvalues_precon_diff}
Let the matrices $\mathcal{A}$ and $\widetilde{\mathcal{A}}_{\text{mg}}$ be defined as in 
\cref{eq_SGFEM_diffusion_matrix} and \cref{eq_block_precon}, respectively. Then, the estimate
\begin{align}
\label{eq_ev_bounds_SGFE_diff_approx_diff}
\hat{\delta} \leq \frac{\vecg{v}^{\text{T}}\mathcal{A} \, \vecg{v}}
{\vecg{v}^{\text{T}} \widetilde{\mathcal{A}}_{\text{mg}} \, \vecg{v}} \leq \hat{\Delta}, \quad 
\forall \vecg{v} \in \RR^{\nPC \nFE_u} \backslash \{\vecg{0}\},
\end{align}
holds with $\hat{\delta} := (\underline{\nu}_0 - \sqrt{3} \,\sigma_\nu \,
\chi)\delta$ and $\hat{\Delta} :=(\overline{\nu}_0 + \sqrt{3}\,\sigma_\nu \,\chi)\Delta$.
\end{lemma}
\begin{proof}
We start with the upper bound:
\begin{align}
\label{eq_bounds_SGFE_diff_upper}
\lambda_{\text{max}}\left(\widetilde{\mathcal{A}}_{\text{mg}}^{-1} \, \mathcal{A} \right) &= 
\max_{\vecg{v}\in\RR^{\nPC \nFE_u} \backslash \{\vecg{0}\}}
\frac{\vecg{v}^{\text{T}}\left(I \otimes A_0 + \sum_{\iKL=1}^{\nKL}G_\iKL \otimes A_\iKL \right) \vecg{v}}
{\vecg{v}^{\text{T}}\left(I \otimes \widetilde{A}_{\text{mg}} \right) \vecg{v}} \\
\notag
&\leq \lambda_{\text{max}}\left(I \otimes \widetilde{A}_{\text{mg}}^{-1}A_0 \right) + 
\sum_{\iKL=1}^\nKL
\lambda_{\text{max}}\left(G_m \otimes \widetilde{A}_{\text{mg}}^{-1}A_\iKL \right).
\end{align}
Considering the remaining terms, estimates for the eigenvalues of $I \otimes \widetilde{A}_{\text{mg}}^{-1}A_0$ 
are given in \cref{eq_upper_bound_mean_Diff_approxDiff}. Bounds for the eigenvalues of $G_m \otimes \widetilde{A}_{\text{mg}}^{-1}A_\iKL$, for $\iKL=1,\dots,\nKL$, follow from \cref{eq_upper_bound_mean_Diff_approxDiff,ev_bounds_G_m}:
\begin{equation}
-\sqrt{3} \, \sigma_\nu \, \chi_\iKL \, \Delta
\leq \frac{\vecg{v}^\text{T}\left(G_\iKL \otimes A_m \right)\vecg{v}}
{\vecg{v}^\text{T} \left( I \otimes \widetilde{A}_{\text{mg}} \right) \vecg{v}}
\leq \sqrt{3} \, \sigma_\nu \, \chi_\iKL \, \Delta, \quad \vecg{v} \in \RR^{\nPC \nFE_u} \backslash \{\vecg{0}\}.
\end{equation}
A subsequent summation over $\iKL$ yields the desired estimate
\begin{equation}
\label{eq_ev_bounds_SGFE_diff_4} 
\sum_{\iKL=1}^\nKL \lambda_{\text{max}}(G_m \otimes \widetilde{A}_{\text{mg}}^{-1}A_\iKL)
\leq \sqrt{3} \, \sigma_\nu \, \chi \, \Delta. 
\end{equation}
The upper bound $(\overline{\nu}_0 + \sqrt{3}\,\sigma_\nu \,\chi)\Delta$ then follows when using 
\cref{eq_upper_bound_mean_Diff_approxDiff,eq_ev_bounds_SGFE_diff_4} in \cref{eq_bounds_SGFE_diff_upper}.

To derive a lower bound, we do not proceed in the same way as for the upper one but use a specific product 
representation as in \cite[Lemma 3.2]{Pultarova2015}:
\begin{align}
\lambda_{\text{min}}\left(\widetilde{\mathcal{A}}_{\text{mg}}^{-1} \, \mathcal{A}\right) &= 
\min_{\vecg{v}\in\RR^{\nPC \nFE_u} \backslash \{\vecg{0}\}}
\frac{\vecg{v}^{\text{T}}\left(I \otimes A_0 + \sum_{\iKL=1}^{\nKL}G_\iKL \otimes A_\iKL \right) \vecg{v}}
{\vecg{v}^{\text{T}}\left(I \otimes \widetilde{A}_{\text{mg}} \right) \vecg{v}} \\
\notag
&= \min_{\vecg{v} \in \RR^{\nPC \nFE_u} \backslash \{\vecg{0}\}} \underbrace{\vecg{v}^{\text{T}} \frac{\left(\mathcal{I} + \sum_{\iKL=1}^{\nKL} 
G_{\iKL} \otimes A_0^{-1} A_{\iKL} \right)\vecg{v}}{\vecg{v}^{\text{T}}\vecg{v}}}_{(\text{I})} \underbrace{\frac{\vecg{v}^{\text{T}}\vecg{v}}
{\vecg{v}^{\text{T}}\left(I \otimes  A_0^{-1}\widetilde{A}_{\text{mg}} \right)\vecg{v}}}_{(\text{II})},
\end{align}
where $\mathcal{I}$ denotes the identity matrix of size $\nPC \nFE_u$.
If the matrices contained in (I) and (II) are symmetric non-negative definite, we can bound $\lambda_{\text{min}}(\widetilde{\mathcal{A}}_{\text{mg}}^{-1}\mathcal{A})$ by the product of the minimum 
eigenvalues of those matrices \cite[Chapter 9]{MarshallOlkin1979}. Both matrices are symmetric due to their 
symmetric building blocks. We are going to verify positive definiteness, starting with (I):
\begin{align}
\label{eq_stoch_diff_mat_lower_bound_end}
\min_{\vecg{v} \in \RR^{\nPC \nFE_u} \backslash \{\vecg{0}\}} \frac{\vecg{v}^{\text{T}}\left(\mathcal{I} + 
\sum_{\iKL=1}^{\nKL} 
G_{\iKL} \otimes A_0^{-1} A_{\iKL} \right)\vecg{v}}{\vecg{v}^{\text{T}}\vecg{v}} 
&\geq 1 + \sum_{\iKL=1}^{\nKL} \min_{\vecg{v} \in \RR^{\nPC \nFE_u} \backslash \{\vecg{0}\}} \frac{\vecg{v}^{\text{T}}\left( 
 G_{\iKL} \otimes A_0^{-1} A_{\iKL} \right)\vecg{v}}{\vecg{v}^{\text{T}}\vecg{v}}  \\
\notag
&\hspace{-0.58cm}\overset{\cref{eq_mean_Diff_Diff},\cref{ev_bounds_G_m}}{\geq} \underline{\nu}^{-1}_0(\underline{\nu}_0 - \sqrt{3} \,\sigma_\nu \,
\chi)\overset{\cref{eq_assumption_KL}}{>} 0.
\end{align}
Since the bound in \cref{eq_stoch_diff_mat_lower_bound_end} is positive, the matrix in (I) is positive 
definite. The matrix in (II) is positive definite as well because all of its building blocks are positive definite, see 
\cref{eq_upper_bound_mean_Diff_approxDiff}. We can therefore 
apply the product bound and \cref{eq_stoch_diff_mat_lower_bound_end} to derive
\begin{equation}
\label{eq_bounds_SGFE_diff_lower}
\lambda_{\text{min}}\left(\widetilde{\mathcal{A}}_{\text{mg}}^{-1} \, \mathcal{A}\right)
 \geq \underline{\nu}^{-1}_0(\underline{\nu}_0 - \sqrt{3} \,\sigma_\nu \,
\chi) \cdot \min_{\vecg{v} \in \RR^{\nPC \nFE_u} \backslash \{0\}} 
\frac{\vecg{v}^{\text{T}}\left(I\otimes \widetilde{A}_{\text{mg}}^{-1} A_0\right)\vecg{v}}
{\vecg{v}^{\text{T}}\vecg{v}}.
\end{equation}
Using the left bound of \cref{eq_upper_bound_mean_Diff_approxDiff} in
\cref{eq_bounds_SGFE_diff_lower} then yields the lower bound $(\underline{\nu}_0 - \sqrt{3} \,\sigma_\nu \,
\chi)\delta$.
\end{proof}

%% file: ea_Schur_complement.tex
\subsection{The preconditioned SGFE Schur complement}
\label{subsec_precon_Schur}
The eigenvalue bounds are derived as in \cref{lemma_eigenvalues_precon_diff} by additionally relying on the inf-sup stability of the FE approximation, see also \cite[Lemma 4.6]{Mueller2018}.
\begin{lemma}
\label{lemma_eigenvalues_precon_Schur}
Let $\mathcal{S} = -\mathcal{B}  \mathcal{A}^{-1} \mathcal{B}^{\text{T}}$ be defined 
with building blocks $\mathcal{A}$ and $\mathcal{B}$ according to \cref{eq_SGFEM_diffusion_matrix} and 
\cref{eq_SGFEM_gradient_matrix}. Further, let $\widetilde{\mathcal{S}}_p = I \otimes D_p$ according to 
\cref{eq_block_precon}. Then, the following estimate holds:
\begin{equation}
\label{eq_ev_bounds_SGFE_Schur_approx_Schur}
\frac{\theta^2 \gamma^2}{\overline{\nu}_0 + \sqrt{3} \, \sigma_\nu \, \chi} \leq \frac{\vecg{q}^\text{T} \, \mathcal{B}  \mathcal{A}^{-1} \mathcal{B}^{\text{T}} \, \vecg{q}}
{\vecg{q}^\text{T} \left( I \otimes D_p \right) \vecg{q}} \leq  \frac{\Theta^2}{\underline{\nu}_0 - \sqrt{3} \, \sigma_\nu \, \chi}, \quad \forall \vecg{q} \in\RR^{\nPC \nFE_p}\backslash\{\vecg{0},\vecg{1}\}.
\end{equation}
\end{lemma}
\begin{proof}
We start by considering
\begin{equation}
\label{eq_ev_bounds_SGFE_Schur_2} 
(I \otimes A)^{-1} \mathcal{A}
= I \otimes A^{-1}A_0 + \sum_{\iKL=1}^\nKL G_\iKL \otimes A^{-1}A_\iKL.
\end{equation}
Using \cref{eq_mean_Diff_Diff,ev_bounds_G_m} with \cref{eq_ev_bounds_SGFE_Schur_2}, we derive
\begin{align}
\label{eq_ev_bounds_SGFE_Schur_5} 
\underline{\nu}_0 -  \sqrt{3} \, \sigma_\nu \, \chi  \leq 
\frac{\vecg{v}^{\text{T}}(I\otimes A)^{-1} \vecg{v}}
{\vecg{v}^{\text{T}} \mathcal{A}^{-1} \vecg{v}} &\leq \overline{\nu}_0 + \sqrt{3} \, \sigma_\nu \, \chi 
\end{align}
for all $\vecg{v} \in \RR^{\nPC \nFE_u} \backslash \{\vecg{0}\}$. Proceeding as in the proof of \cite[Lemma 2.4]{SilvesterWathen1994}, we derive
\begin{align}
\vecg{q}^\text{T} \widetilde{\mathcal{S}}_p^{-1}\mathcal{B}  \mathcal{A}^{-1} \mathcal{B}^{\text{T}} \vecg{q} &\overset{\cref{eq_ev_bounds_SGFE_Schur_5}}{\leq} (\underline{\nu}_0 - \sqrt{3} \, \sigma_\nu \, \chi)^{-1} \, \vecg{q}^\text{T} \widetilde{\mathcal{S}}_p^{-1}\mathcal{B}  (I\otimes A)^{-1} \mathcal{B}^{\text{T}}\vecg{q} \\
\notag
&\overset{\cref{eq_bound_Schur_diagMass}}{\leq} (\underline{\nu}_0 - \sqrt{3} \, \sigma_\nu \, \chi)^{-1} \Theta^2 \vecg{q}^\text{T} \vecg{q}
\end{align}
for $\vecg{q} \in \RR^{\nPC \nFE_p} \backslash \{\vecg{1}\}$ which is the upper bound in the assertion. The lower bound is derived analogously using the upper bound in \cref{eq_ev_bounds_SGFE_Schur_5} and the lower bound in \cref{eq_bound_Schur_diagMass}.
\end{proof}

%% file: ea_pseudo_Schur_complement.tex
\subsection{The preconditioned approximate SGFE Schur complement}
\label{subsec_precon_pseudo_Schur}
We call $\mathcal{B}\widetilde{\mathcal{A}}_{\text{mg}}^{-1} \mathcal{B}^{\text{T}}$ the approximate SGFE Schur complement. The matrix looks like the original Schur complement, but the Laplacian is replaced by its preconditioner. Bounds on the eigenvalues of the preconditioned approximate Schur complement are derived in the following lemma.
\begin{lemma}
\label{lemma_eigenvalues_precon_quasi_Schur}
Let $\widetilde{\mathcal{A}}_{\text{mg}}=I \otimes \widetilde{A}_{\text{mg}}$ and $\widetilde{S}_p=I \otimes D_p$ according to 
\cref{eq_block_precon}. Then, the following estimate holds:
\begin{equation}
\label{eq_bounds_SGFE_quasi_Schur_approx_Schur}
\delta \, \theta^2 \gamma^2 \leq \frac{\vecg{q}^{\text{T}} \mathcal{B}\widetilde{\mathcal{A}}_{\text{mg}}^{-1}
\mathcal{B}^{\text{T}}\vecg{q}}
{\vecg{q}^{\text{T}}\widetilde{S}_p \vecg{q}} 
\leq \Delta \, \Theta^2, \quad \forall \vecg{q} \in \RR^{\nPC \nFE_p} \backslash \{\vecg{0},\vecg{1}\}.
\end{equation}
\end{lemma}
\begin{proof}
We follow the steps of the proof of \cite[Lemma 2.4]{SilvesterWathen1994}. Inserting the definitions yields
\begin{equation}
\label{eq:ev_bounds_approx_Schur_arbit_I}
\vecg{q}^{\text{T}} \widetilde{\mathcal{S}}_p^{-1} \mathcal{B} \widetilde{\mathcal{A}}_{\text{mg}}^{-1} \mathcal{B}^{\text{T}}  \vecg{q} = \vecg{q}^{\text{T}}
(I \otimes D_p^{-1} B \widetilde{A}_{\text{mg}}^{-1} B^{\text{T}} )\vecg{q}
\end{equation}
for all $\vecg{q} \in \RR^{\nPC \nFE_p}$. As the identity matrix does not influence spectral bounds for the expression, the considerations can be reduced to the FE matrices on the right-hand side of \cref{eq:ev_bounds_approx_Schur_arbit_I}:
\begin{align*}
\vecg{q}^{\text{T}} D_p^{-1} B \widetilde{A}_{\text{mg}}^{-1} B^{\text{T}} \vecg{q} \overset{\cref{eq_upper_bound_Diff_approxDiff}}{\leq} \Delta \, \vecg{q}^{\text{T}}
D_p^{-1} B A^{-1} B^{\text{T}} \vecg{q} \overset{\cref{eq_bound_Schur_diagMass}} \leq \Delta\, \Theta^2 \, \vecg{q}^{\text{T}}
\vecg{q},
\end{align*}
for $\vecg{q} \in \RR^{\nFE_p}\backslash \{\vecg{1}\}$. The upper bound in the assertion immediately follows and the lower one is derived analogously using the lower bounds in \cref{eq_upper_bound_Diff_approxDiff,eq_bound_Schur_diagMass}.
\end{proof}
Based on the bounds in the \cref{lemma_eigenvalues_precon_diff,lemma_eigenvalues_precon_Schur,lemma_eigenvalues_precon_quasi_Schur}, existing convergence results are utilized in \cref{sec_iterative_solvers} to predict the convergence behavior of the iterative solvers that are presented.

%% file: iterative_solvers.tex
\section{Iterative solvers}
\label{sec_iterative_solvers}
As solvers for the coupled SGFE system, we compare a preconditioned MINRES method \cite{PaigeSaunders1975} with a BPCG method \cite{BramblePasciak1988}. A similar comparison for Stokes flow with deterministic data can be found in~\cite{PetersEtAl2005}. We utilize existing results to make predictions about the convergence behavior of the linear solvers based on the eigenvalue bounds derived in \cref{sec_ev_analysis}. Finally, we comment on the complexity and implementation of the two solvers.

\subsection{The MINRES method}
\label{subsec_is_minres}
The system matrix $\mathcal{C}$ of the SGFE Stokes problem \cref{eq_block_matrix_form} is symmetric and indefinite, see \cref{eq_block_factorizations}. The MINRES method is a standard Krylov subspace solver for systems of equations with such a structure. Because MINRES relies on the symmetry of the problem, it requires a symmetric preconditioner which must be positive definite as well. This excludes the block triangular preconditioner $\mathcal{P}_2$ in \cref{eq_block_precon}. Hence, we will use MINRES solely in combination with the block diagonal preconditioner $\mathcal{P}_1$.

MINRES is a short-recurrence method and applications of $\mathcal{C}$ and $\mathcal{P}_1^{-1}$ are the only matrix-vector operations necessary per iteration \cite[Algorithm~4.1]{ElmanEtAl2014}. It is worth noticing that MINRES convergence does not rely on a scaling of the involved preconditioners and the solver can thus be considered parameter-free.

A standard convergence result for MINRES, as given in \cite[Theorem 4.14]{ElmanEtAl2014}, can be applied when eigenvalue bounds for the preconditioned system matrix are available. In our setting, eigenvalue bounds for $\mathcal{P}_1^{-1}\mathcal{C}$ are thus needed. They can be constructed based on the bounds derived in the \cref{lemma_eigenvalues_precon_diff,lemma_eigenvalues_precon_quasi_Schur} using \cite[Theorem 4.1]{ErnstEtAl2009}, see also \cite[section 4.2.4]{Mueller2018}.

As a consequence, the convergence behavior of the MINRES method combined with the preconditioner $\mathcal{P}_1$ is influenced by the modeling and discretization parameters contained in the bounds in the \cref{lemma_eigenvalues_precon_diff,lemma_eigenvalues_precon_quasi_Schur}.
The iteration counts should thus be asymptotically independent of the parameters that do not occur in those
bounds, i.e.\ the mesh width $h$, the degree of the polynomial chaos $\stochasticDegree$ and 
the dimension of the random input $\nKL$.~Nevertheless, the bounds still depend on
the mean field $\nu_0(x)$ and the process standard deviation 
$\sigma_\nu$.

\subsection{The Bramble-Pasciak conjugate gradient method}
\label{subsec_is_bpcg}
When the triangular preconditioner $\mathcal{P}_2$ in \cref{eq_block_precon} is applied instead of 
$\mathcal{P}_1$, the preconditioned system matrix $\mathcal{P}_2^{-1} \mathcal{C}$ is nonsymmetric. There does 
not exist a short recurrence for generating an orthogonal Krylov subspace basis for every nonsymmetric matrix
\cite{FaberManteuffel1984}. As an alternative, one can use a short-recurrence and construct a Krylov subspace 
basis which is not orthogonal, thereby loosing the minimization property of the method, see the Bi-CGSTAB 
algorithm \cite{VanDerVorst1992}. The other option is to maintain the orthogonality of the basis at the 
expense of a full-recurrence orthogonalization procedure, as it is done for GMRES \cite{SaadSchultz1986}. 

There are special nonsymmetric problems for which an orthogonal Krylov subspace basis can be 
constructed with short recurrence \cite{FaberManteuffel1984}. The block triangular preconditioned system matrix
$\mathcal{P}_2^{-1} \mathcal{C}$ is such a case due to the manipulations introduced to $\mathcal{P}_2$ in the beginning of \cref{sec_ev_analysis}. This is shown in the following lemma.
\begin{lemma}
\label{lemma_ipcg_conditions}
Let $a$ in \cref{eq_block_precon} be set to $a= a_\delta \hat{\delta}$, with 
$0 < a_\delta < 1$ and $\hat{\delta}$ as in \cref{lemma_eigenvalues_precon_diff}. Then, the matrix 
\begin{align}
\label{eq_inner_product_matrix}
\mathcal{H} := \begin{bmatrix} \mathcal{A} - a_{\delta} \hat{\delta}\widetilde{\mathcal{A}}_{\text{mg}} & 0 \\ 0 & 
\widetilde{\mathcal{S}}_p \end{bmatrix},
\end{align}
with $\widetilde{\mathcal{A}}_{\text{mg}}$ and $\widetilde{\mathcal{S}}_p$ as in 
\cref{eq_block_precon},  defines an inner product. Moreover, the triangular preconditioned system 
matrix $\mathcal{P}_2^{-1} \mathcal{C}$, with $\mathcal{C}$ and $\mathcal{P}_2$ according to 
\cref{eq_block_matrix_form,eq_block_precon}, respectively, is $\mathcal{H}$-symmetric and $\mathcal{H}$-positive definite, 
i.e.
\begin{align}
\label{eq_inner_product_cond_2}
\hspace{2cm} &\mathcal{H} \mathcal{P}_2^{-1} \mathcal{C} = (\mathcal{P}_2^{-1}\mathcal{C})^T\mathcal{H}, \\
\label{eq_inner_product_cond_1}
&\vecg{w}^{\text{T}} \mathcal{H} \mathcal{P}_2^{-1} \mathcal{C} \vecg{w} > 0, \hspace{2cm} \forall 
\vecg{w} \in \RR^{\nPC \nFE} \backslash\{\vecg{0}\}.
\end{align}
\end{lemma}
\begin{proof}
First of all, we make sure that the matrix $\mathcal{H}$ defines an inner product, i.e.\ check that it is 
symmetric and positive definite. Subsequently, we establish  \cref{eq_inner_product_cond_2} and \cref{eq_inner_product_cond_1}.

The symmetry of $\mathcal{H}$ is verified by noting that the building blocks of the matrices on the diagonal are symmetric. To verify positive definiteness of $\mathcal{H}$, we consider its diagonal blocks separately. The lower diagonal block $\widetilde{\mathcal{S}}_p = I \otimes D_P$ is positive definite because its Kronecker factors are positive definite, see \cite[Section 3.5]{ElmanEtAl2014}. The investigation of the upper diagonal block is more intricate: The matrix $\mathcal{A} - a_{\delta} \hat{\delta}\widetilde{\mathcal{A}}_{\text{mg}}$ is positive definite if
\begin{equation}
\label{eq_condition_pd_upper}
\frac{\vecg{v}^{\text{T}}(\mathcal{A} - a_{\delta} \hat{\delta}\widetilde{\mathcal{A}}_{\text{mg}})
\vecg{v}}{\vecg{v}^{\text{T}} \vecg{v}} = \frac{\vecg{v}^{\text{T}}\mathcal{A} \vecg{v}}{\vecg{v}^{\text{T}} \vecg{v}}
 - \frac{\vecg{v}^{\text{T}}a_{\delta} \hat{\delta} \widetilde{\mathcal{A}}_{\text{mg}} \vecg{v}}{\vecg{v}^{\text{T}} \vecg{v}}
 >  0 \quad \forall \vecg{v} \in \RR^{\nPC \nFE_u} \backslash \{\vecg{0}\}. 
\end{equation}
We can express this condition equivalently by rearranging \cref{eq_condition_pd_upper}:
\begin{equation}
\label{eq_condition_pd_bpcg}
\min_{\vecg{v}\in \RR^{\nPC \nFE_u} \backslash \{\vecg{0}\}} 
\frac{\vecg{v}^{\text{T}}\mathcal{A} \vecg{v}}{\vecg{v}^{\text{T}} a_{\delta} \hat{\delta}\widetilde{\mathcal{A}}_{\text{mg}}
\vecg{v}} > 1.
\end{equation}
From \cref{eq_ev_bounds_SGFE_diff_approx_diff} and $0 < a_\delta < 1$, we can deduce $\lambda_{\text{min}}(a_{\delta}^{-1} \hat{\delta}^{-1} \widetilde{\mathcal{A}}_{\text{mg}}^{-1} \mathcal{A}) \geq a_\delta^{-1} > 1$, which eventually ensures the positive definiteness of $\mathcal{A} - a_{\delta} \hat{\delta}\widetilde{\mathcal{A}}_{\text{mg}}$ and thereby the positive 
definiteness of $\mathcal{H}$.

In order to see that \cref{eq_inner_product_cond_2} holds, we perform the matrix-matrix products on both sides of the equals sign and ascertain that the results are identical. Condition \cref{eq_inner_product_cond_1} is established based on a matrix 
factorization in the form of a congruence transform, see \cite[Lemma 4.9]{Mueller2018}.
\end{proof}
As a consequence, the nonsymmetric system associated with $\mathcal{P}_2^{-1} \mathcal{C}$ can be solved with a CG method in the $\mathcal{H} \mathcal{P}_2^{-1} \mathcal{C}$-inner product. However, this approach also has disadvantages: Firstly, the preconditioner $\widetilde{\mathcal{A}}_{\text{mg}}$ must be scaled such that \cref{eq_condition_pd_bpcg} holds, i.e.\ the scaling parameter $a$ must be chosen properly. This can be done using the analytically derived bound \cref{eq_ev_bounds_SGFE_diff_approx_diff}, as was shown in \cref{lemma_ipcg_conditions}. When this bound is too pessimistic, using the scaling based on the analytical estimate might lead to higher iteration counts than using the \textit{correct} scaling. In such cases, determining the scaling by computing $\lambda_{\text{min}}( \widetilde{\mathcal{A}}_{\text{mg}}^{-1} \mathcal{A})$ numerically could be beneficial, though resulting in increased computational costs. We investigate the influence of the scaling on the convergence behavior of the BPCG method in \cref{table_varying_scaling}.~Secondly, the naive BPCG procedure is formulated in the~$\mathcal{H} \mathcal{P}_2^{-1} \mathcal{C}$-inner product.~Additional matrix-vector operations could arise from the evaluation of quantities in that inner product. Fortunately, CG methods have several properties which can be exploited to reformulate the algorithm \cite{AshbyEtAl1990}. Using these reformulations, it is possible to derive a BPCG algorithm which needs only one extra matrix-vector operation per iteration compared to preconditioned MINRES: a multiplication with $\mathcal{B}$, which originates from the definition of $\mathcal{P}_2$ in \cref{eq_block_precon}, see \cite{PetersEtAl2005}. As the matrix $\mathcal{B}$ is block diagonal, however, this operation is cheap compared to a multiplication with $\mathcal{A}$, which is not block diagonal. Because of this feature, the BPCG method is particularly interesting in the SGFE setting.

A standard convergence result for CG methods such as \cite[Theorem 9.4.12]{Hackbusch1994} relies on eigenvalue bounds for the preconditioned system matrix. For the block triangular preconditioned matrix $\mathcal{P}_2^{-1}\mathcal{C}$, these bounds can be constructed based on \cite[Theorem 4.1]{Zulehner2001}, using the bounds in \cref{lemma_eigenvalues_precon_diff} and a scaled version of the bounds in \cref{lemma_eigenvalues_precon_Schur}, see \cite[section 4.2.5]{Mueller2018}. Following the same line of argumentation as for the MINRES method in \cref{subsec_is_minres}, we use this convergence result to predict the convergence behavior of the BPCG solver with respect to different problem parameters: We claim that the convergence of the method is asymptotically independent of the parameters which do not occur in the bounds in the \cref{lemma_eigenvalues_precon_diff,lemma_eigenvalues_precon_Schur}. Consequently, the iteration counts should be asymptotically independent of the spatial mesh width $h$, the chaos degree $\stochasticDegree$ and the dimension $\nKL$ of the random input. The mean field $\nu_0(x)$ and the process standard deviation $\sigma_\nu$ remain in the bounds.

%% file: numerical_examples.tex
\section{Numerical test case}
\label{sec_numerical_examples}
We investigate the convergence behavior of the two iterative methods discussed in \cref{sec_iterative_solvers} 
by the help of a numerical example. Furthermore, we check the theoretical predictions and 
compare the performance of the solvers.

In order to parametrize the uncertain viscosity properly with a moderate number of random variables, we use the separable
exponential covariance function $C_\nu: \spatialDomain \times \spatialDomain \to \RR$ defined by $C_\nu(x,z) = \sigma^2_\nu \, e^{-\abs{x_1-z_1}/b_1 - \abs{x_2-z_2}/b_2}
$, where the scalars $b_1$ and $b_2$ are the correlation lengths in the $x_1$ and $x_2$ direction, respectively. 
The eigenpairs of the two-dimensional KLE can be constructed by combining the eigenpairs 
of two one-dimensional equations, which have an analytical representation \cite[section 5.3]{GhanemSpanos1991}, 
\cite[section 7.1]{LordEtAl2014}.

As a numerical test case, we consider a regularized lid-driven cavity \cite[section 3.1]{ElmanEtAl2014} on the 
unit square $\spatialDomain=[-0.5,0.5]\times[-0.5,0.5]$. We use no-slip conditions for the velocity field 
everywhere on the domain boundary except at the top lid, where we prescribe a parabolic flow profile 
$u(x) = (1-16x_1^4,0)^T$. The volume forces $f(x)$ are set to zero. If not specified otherwise, we use 
the following default parameter values for the simulations:
\begin{align}
\label{eq_default_parameter_set}
h=0.01, \quad \stochasticDegree=2, \quad \nKL= 10, \quad \nu_0 = 1, \quad \sigma_\nu = 0.1, \quad b_1=b_2=1.
\end{align}

The multigrid method is part of both of our preconditioners in \cref{eq_block_precon}. For the following 
simulations, we use the algebraic multigrid implemented in the IFISS package 
\cite{ElmanEtAl2014b} with two point Gauss-Seidel pre-and post-smoothing sweeps. Numerical experiments in 
\cite[section 2.5]{ElmanEtAl2014} suggest that relation 
\cref{eq_upper_bound_Diff_approxDiff} is then fulfilled with constants $\delta \approx 0.85$ and 
$\Delta \approx 1.15$. All other computations are carried out in our own finite element implementation 
\cite{Ullmann2016} in MATLAB. 

In the following, we compare the iterative methods discussed in \cref{sec_iterative_solvers}. We consider the $\mathcal{P}_1$-pre\-con\-di\-tio\-ned MINRES ($\mathcal{P}_1$-MINRES) and the $\mathcal{P}_2$-preconditioned BPCG, with the analytical scaling $a=\hat{\delta}$ according to \cref{lemma_eigenvalues_precon_diff} ($\mathcal{P}^{ana}_2$-BPCG) and with the scaling computed by numerically approximating $\lambda_{\text{min}}(\widetilde{\mathcal{A}}_{\text{mg}}^{-1}\mathcal{A})$ 
($\mathcal{P}^{num}_2$-BPCG). In fact, we should set $a < \hat{\delta}$, see \cref{lemma_ipcg_conditions}. The results below, however, suggest that the solver is robust with respect to small variations of $a$. We look at the iteration counts necessary to reduce the Euclidean norm of the relative residual below~$10^{-6}$, i.e.~the numbers $n$ such that $\norm{\vecg{r}^{(n)}} = \norm{\vecg{b}-\mathcal{C}\vecg{z}^{(n)}} \leq 10^{-6} \norm{\vecg{b}}$. Here, $\vecg{z}^{(n)}$ denotes the $n$-th iterate. We use the zero vector as initial guess, i.e.~$\vecg{z}^{(0)}=\vecg{0}$.

At first, we want to assess the influence of the scaling factor $a$ introduced in \cref{sec_ev_analysis} on the convergence of the BPCG method. Therefore, we define the value $a^*$ such that $\mathcal{A}-a^* \widetilde{\mathcal{A}}_{\text{mg}}$ is bordering positive definiteness and solve the reference problem for different ratios of $a/a^*$. We calculate $a^*$ with MATLAB's (version 8.6.0) numerical eigensolver \texttt{eigs}, i.e.\ $a^* \approx \lambda_{\text{min}}(\widetilde{\mathcal{A}}_{\text{mg}}^{-1}\mathcal{A})$. Corresponding iteration counts for $\mathcal{P}_2^{num}$-BPCG are displayed in \cref{table_varying_scaling}. We note that the minimum iteration count of 27 is attained for $a/a^*=1$. The analytical bound according to \cref{lemma_eigenvalues_precon_diff}, which implies $a/a^*=0.48$, leads to an increased iteration count of 36.~This means that the analytical bound is too pessimistic in the sense of overfulfilling the positive definiteness requirement. For $a/a^*>1$, we can not theoretically guarantee that the algorithm converges. Still, we could not observe divergent behavior in our experiments. From the results in \cref{table_varying_scaling}, we deduce that the iteration counts are robust with respect to moderate variations around the optimal scaling $a = a^*$.
\begin{table}[tbhp]
\caption{Iteration counts for $\mathcal{P}^{num}_2$-BPCG using different values of the relative scaling 
$a/a^*$.~The boldfaced ratio corresponds to the analytical bound according to 
\cref{lemma_eigenvalues_precon_diff}.}
\label{table_varying_scaling}
\centering
\begin{tabular}{l c c c c c c c c c c c c c}
\toprule
  $a/a^*$ & 0.1 & 0.2 &  0.4 & $\boldsymbol{0.48}$ & 0.6 & 0.8 & 1.0 & 1.2 & 1.4 & 2.0 & 3  & 5  & 10 \\
  \midrule
  $n$     & 46  & 41  &  38  & 36                  & 34  & 30  & 27  & 27  & 30  & 32  & 38 & 45 & 58 \\ 
\bottomrule
\end{tabular}
\end{table}

In the following, we compare the performance of $\mathcal{P}_1$-MINRES, $\mathcal{P}^{ana}_2$-BPCG and $\mathcal{P}^{num}_2$-BPCG with respect to the variation of different problem parameters. To reduce the computational costs associated with numerically computing the scaling factor for $\mathcal{P}^{num}_2$-BPCG, we solve the associated eigenproblems on the coarsest mesh with mesh size $h=0.1$ in the following examples.~As $h$ does not appear in the corresponding eigenvalue bounds in \cref{eq_ev_bounds_SGFE_diff_approx_diff}, the scaling on the coarse mesh should be close to the scaling on the actual mesh. 

The iteration counts for different values of the mesh size $h$ are displayed in the top left plot of \cref{figure_variation_para}. The $\mathcal{P}_1$-MINRES solver needs the most iterations and the numerically scaled BPCG converges faster than the analytically scaled one. We observe that the iteration counts do not increase 
when refining the mesh but rather decrease moderately for all three methods. This matches the predictions made 
at the end of the \cref{subsec_is_minres,subsec_is_bpcg}, which suggested that the convergence does not degrade when the mesh is refined.~We are not aware of a theoretical explanation for the decrease in iteration counts. 

\begin{figure}[tbhp]
\centering
\subfloat{\includegraphics[width=0.39\textwidth]{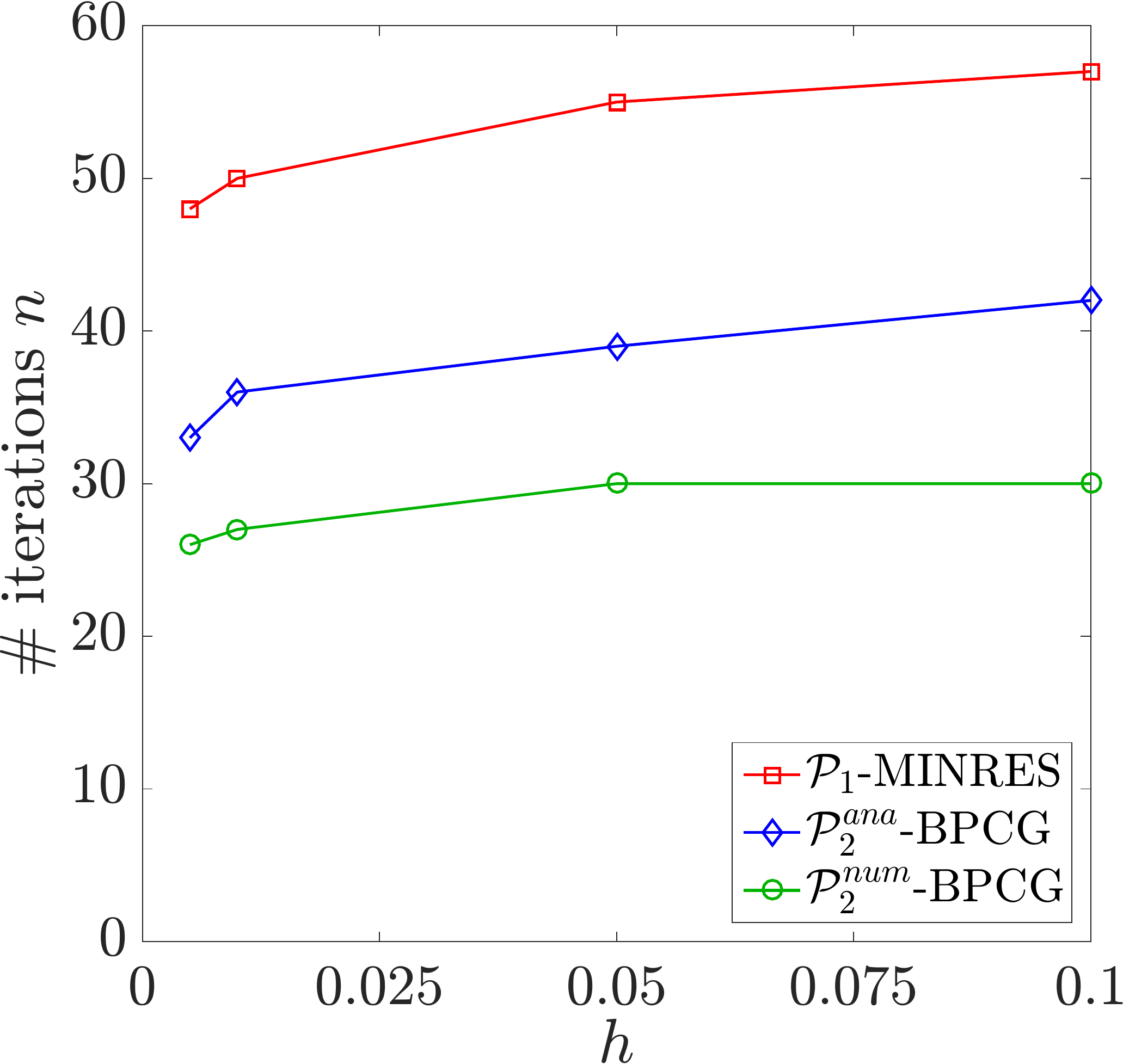}} \hspace{0.1\textwidth}
\subfloat{\includegraphics[width=0.39\textwidth]{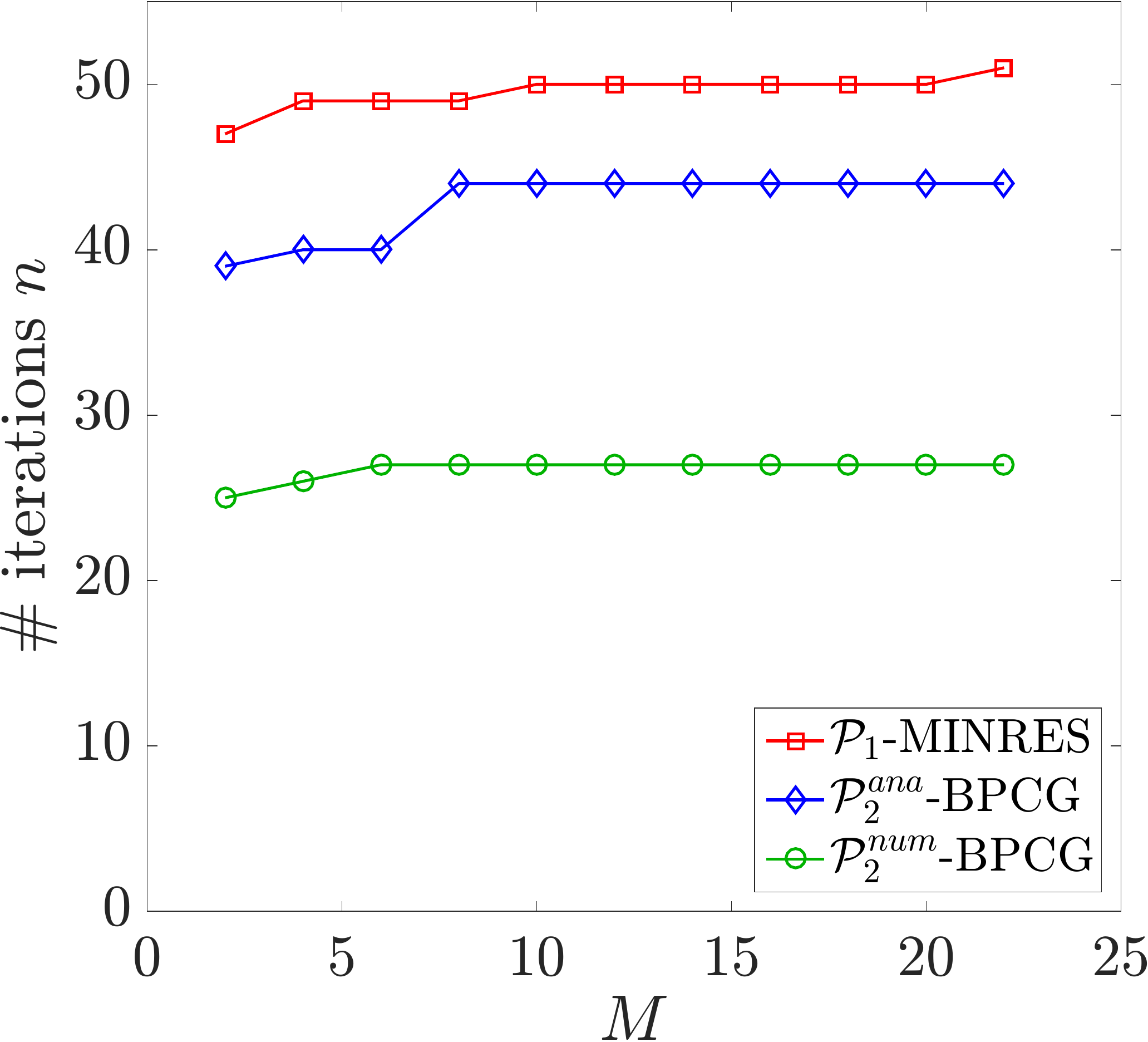}} \\
\subfloat{\includegraphics[width=0.39\textwidth]{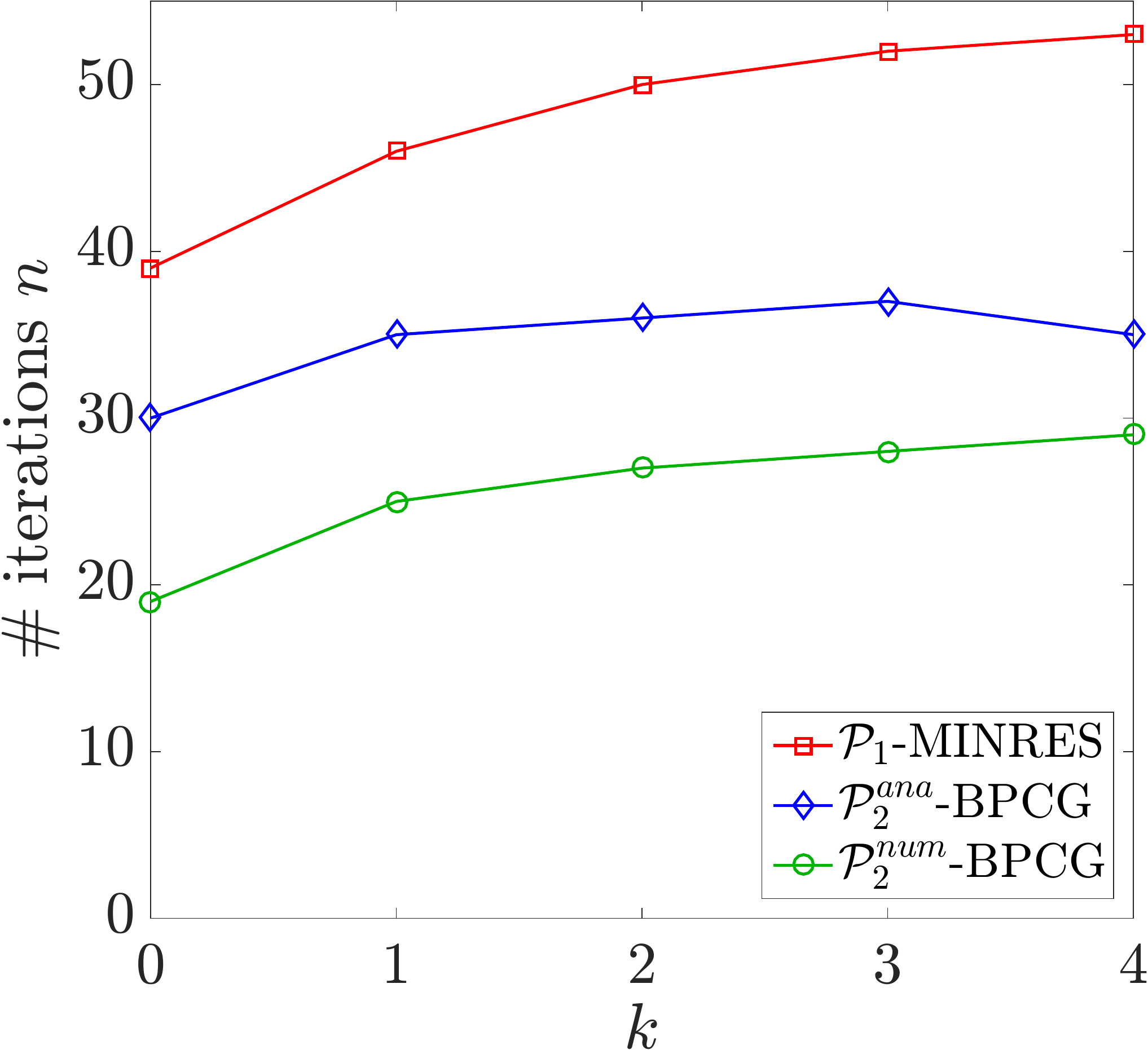}} \hspace{0.1\textwidth}
\subfloat{\includegraphics[width=0.39\textwidth]{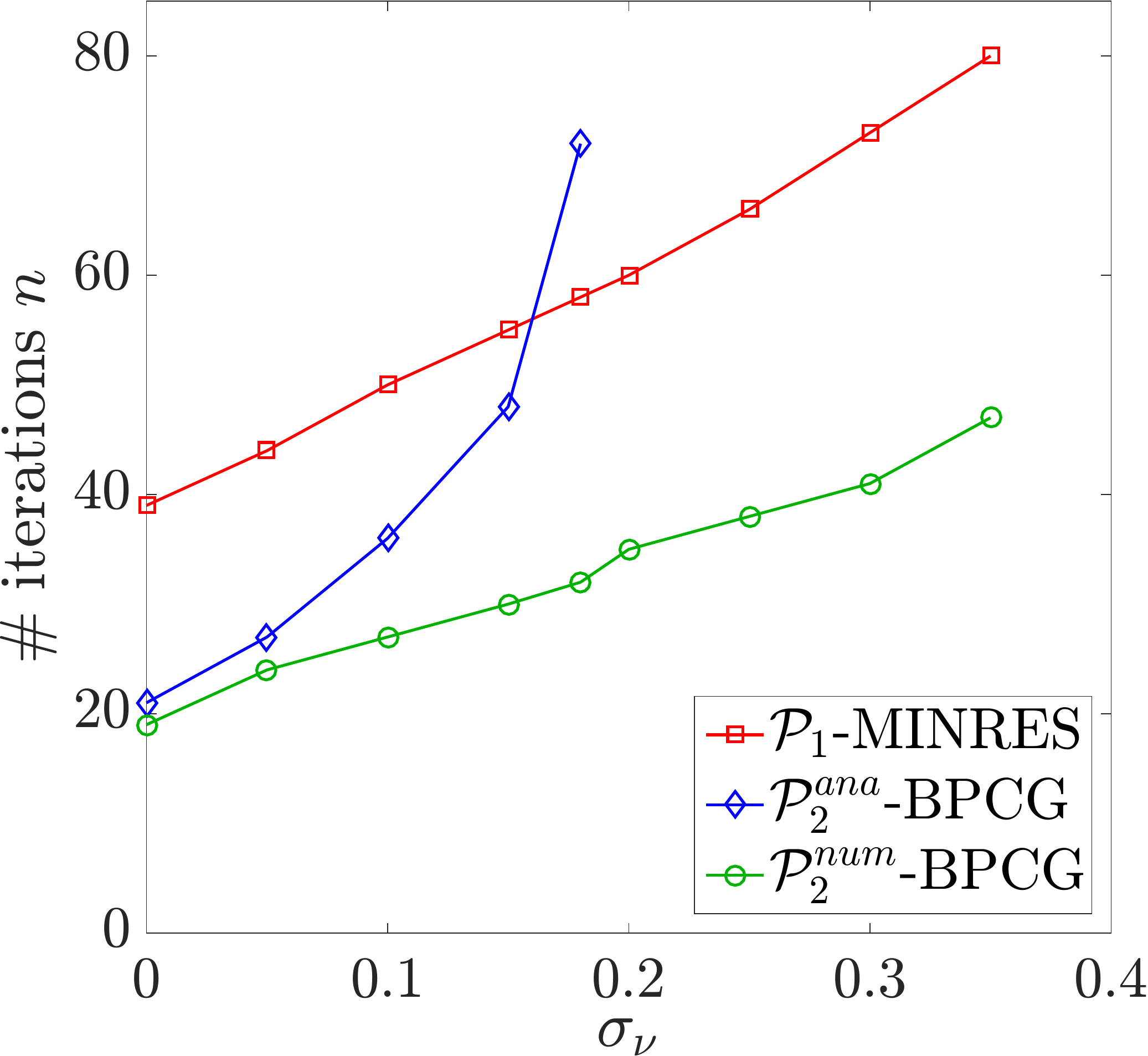}}
\caption{Iteration counts for $\mathcal{P}_1$-MINRES (red), $\mathcal{P}^{ana}_2$-BPCG 
(blue) and $\mathcal{P}^{num}_2$-BPCG (green) using different values of the mesh size $h$ (top left), the 
KLE truncation index $\nKL$ (top right), the degree of the polynomial chaos $\stochasticDegree$ (bottom left) and the standard deviation $\sigma_\nu$ (bottom right).}
\label{figure_variation_para}
\end{figure}

The influence of the KLE truncation index $\nKL$ on the convergence of the iterative solvers is displayed in the top right plot of \cref{figure_variation_para}.~When increasing $\nKL$, the iteration counts grow slightly up to a certain threshold and then basically stay constant for all three solvers. Asymptotic independence of $\nKL$ thus occurs as predicted. $\mathcal{P}_1$-MINRES again has the highest iteration counts followed by $\mathcal{P}^{ana}_2$-BPCG and $\mathcal{P}^{num}_2$-BPCG.

In \cref{figure_variation_para}, iteration counts for different values of the chaos degree 
$\stochasticDegree$ are displayed on the bottom left. For almost all considered chaos degrees, the three 
solvers exhibit a slight growth in iteration counts with increasing $k$. For $\mathcal{P}_1$-MINRES and 
$\mathcal{P}^{num}_2$-BPCG, this behavior can be observed up to the considered order where for 
$\mathcal{P}^{ana}_2$-BPCG, the iteration counts decrease at the end. 
The theory did not suggest a dependence on $\stochasticDegree$, so the asymptotic regime is apparently not 
reached yet. We can verify this by looking at the eigenvalues of the stochastic Galerkin matrices 
$G_\iKL$, $\iKL=1,\dots,\nKL$, for $\stochasticDegree=3$: $\lambda_{\text{max}}(G_\iKL) = 1.4915$, 
i.e.\ the eigenvalues do not immediately fill the whole possible range $[-\sqrt{3},\sqrt{3}]$. In fact, for 
$\stochasticDegree = 4$ we get the maximum eigenvalue $\lambda_{\text{max}}(G_\iKL) = 1.570$, which means that 
increasing the degree of the polynomial chaos leads to new, larger eigenvalues, thereby 
increasing the iteration count.~We suspect that the decrease in iteration counts for $\mathcal{P}^{ana}_2$-BPCG has the same cause: For increasing $k$, the analytical scaling improves as it relies on the theoretical bound $\pm\sqrt{3}$.


\Cref{figure_variation_para} (bottom right) contains results for different values of the standard 
deviation~$\sigma_\nu$. We observe that the iteration counts of all three methods noticeably increase with~$\sigma_\nu$. This is due to the fact that $\sigma_\nu$ directly influences the bounds in the \cref{lemma_eigenvalues_precon_diff,lemma_eigenvalues_precon_Schur}. For the analytically scaled $\mathcal{P}^{ana}_2$-BPCG, this influence is particularly pronounced and the algorithm converges merely up to $\sigma_\nu=0.18$. This is because the scaling factor $\hat{\delta}$ in \cref{lemma_eigenvalues_precon_diff} contains the term $\underline{\nu}_0 - \sqrt{3} \,\sigma_\nu \, \chi$ which is also the lower bound for the viscosity in \cref{eq_KL_estiamte}. In the numerical examples, this term becomes negative for $\sigma_\nu >0.18$. Consequently, we can not theoretically guarantee well-posedness of the discrete problem for $\sigma_\nu> 0.18$ as assumption \cref{eq_assumption_KL} no longer holds. Still, computing $\lambda_{\text{min}}(\widetilde{\mathcal{A}}_{\text{mg}}^{-1}\mathcal{A})$ numerically showed that the discrete problem is well-posed up to $\sigma_\nu \approx 0.35$, independent of the chaos degree. Consequently, we need a lower bound tighter than the one in \cref{eq_KL_estiamte} to make $\mathcal{P}^{ana}_2$-BPCG work for $\sigma>0.18$. Yet, for moderate values of $\sigma_\nu$, $\mathcal{P}_1$-MINRES needs the most iterations, followed by $\mathcal{P}^{ana}_2$-BPCG and $\mathcal{P}^{num}_2$-BPCG.


%% file: conclusion.tex
\section{Conclusion}
Based on the SGFE discretization of the Stokes model with random viscosity, we have investigated the efficient solution of the resulting system of equations by means of two suitable preconditioned iterative methods.~We have linked the convergence behavior of the solvers to the eigenvalue bounds of the respective relevant sub-matrices. Because the eigenvalue bounds of $\mathcal P_1^{-1}\mathcal C$ and $P_2^{-1}\mathcal C$ depend on the same 
parameters, we expect the iteration counts of $\mathcal P_1$-MINRES and the $\mathcal P_2$-BPCG methods to depend on these parameters, too. This has been confirmed by numerical experiments. In particular, we have observed the following:
\begin{itemize}[leftmargin=.15in]
\item The mesh size does not have a significant influence on the iteration counts. This is in line with the predictions and results from the use of the spectrally equivalent FE preconditioners.
\item The iteration numbers are asymptotically independent of the KLE truncation index. This is an effect of the conditions imposed on the parametrized random viscosity.
\item There is a slight increase of iteration counts when the discretization of the random parameter domain is refined. This has been explained by the observation that the eigenvalues do not immediately occupy the whole predicted interval, but fill it steadily as the chaos degree grows. Still, as the uniform random variables are bounded and appear linearly in the input, the condition of the SG matrices is bounded independently of the polynomial chaos degree.
\item The condition of the problem critically depends on the variance of the data in relation to its expectation. If the variance becomes too large, the problem even becomes ill-posed. This is reflected by the fact that the convergence deteriorates as the variance increases.  
\end{itemize}
The preconditioner used in the BPCG method contains a scaling to ensure positive definiteness of the matrix $\mathcal{H}$ and $\mathcal{H}$-positive definiteness of the preconditioned system matrix $\mathcal{P}_2^{-1} \mathcal{C}$. The optimal scaling parameter is given by the minimum eigenvalue of the preconditioned SGFE Laplacian, as confirmed by the numerical experiments. Still, the numerical costs associated with this eigenvalue problem are prohibitive.~We have derived an analytical value for the scaling based on knowledge of the parametrized random input. Numerical tests performed with this value resulted in a reduction of the iteration counts by up to a factor of about 1.5 compared to the reference preconditioned MINRES. The only regime where the analytically scaled BPCG performed worse is close to the analytically derived limit of positivity of the viscosity.~We also used the BPCG method with a numerically computed scaling. This approach converged faster than the analytically scaled solver and resulted in a reduction of iteration counts by up to a factor of 2 compared to the preconditioned MINRES method.

Concerning the eigenvalue analysis and the numerical tests, we summarize that the BPCG solver with block 
triangular preconditioner has the same qualitative properties as block diagonal preconditioned MINRES.~When a proper scaling parameter is available, the use of a BPCG solver can result in a noticeable reduction of iteration counts compared to the application of the parameter-free MINRES method with block diagonal preconditioner.~At the same time, applying the block triangular preconditioner is only marginally more expensive.